\documentclass[a4paper, 11pt, reqno]{amsart}
\usepackage[initials]{amsrefs}
\usepackage{xcolor}
\usepackage{amsmath,amscd,amssymb,latexsym, mathrsfs, amsthm}

\evensidemargin=\oddsidemargin
\DeclareMathVersion{can}
\DeclareMathAlphabet{\can}{OT1}{cmss}{m}{n}
\newtheorem{thm}{ Theorem}[section]
\newtheorem{cor}[thm]{Corollary}
\newtheorem{lem}[thm]{Lemma}
\newtheorem{prop}[thm]{Proposition}

\newtheorem{rem}[thm]{Remark}

\numberwithin{equation}{section}
\newcommand{\mc}{\mathcal}

\newcommand{\dis}{\displaystyle}

\newcommand{\A}{\mathbb{A}}

\newcommand{\OO}{\mathcal{O}}

\newcommand{\bF}{\mathbb{F}}

\newcommand{\PP}{\mathbb{P}}
\newcommand{\mb}{\mathbb}

\begin{document}
\title[Average values of $L$-functions in even characteristic]
{Average values of $L$-functions in even characteristic}%

\author[S. Bae]{Sunghan Bae}
\address{\rm Department of Mathematics, KAIST, Daejon 305-701, Korea}
\email{shbae@kaist.ac.kr}

\author[H. Jung]{Hwanyup Jung}
\address{\rm Department of Mathematics Education, Chungbuk National University, Cheongju 361-763, Korea}
\email{hyjung@chungbuk.ac.kr}
\subjclass[2000]{11R11, 11R29, 11R42, 11R58}%
\keywords{L-functions, Class numbers, Quadratic function fields}%

\begin{abstract}
Let $k = \bF_{q}(T)$ be the rational function field over a finite field $\bF_{q}$, where $q$ is a power of $2$.
In this paper we solve the problem of averaging the quadratic $L$--functions $L(s, \chi_{u})$ over fundamental discriminants.
Any separable quadratic extension $K$ of $k$ is of the form $K = k(x_{u})$, where $x_{u}$ is a zero of $X^2+X+u=0$ for some $u\in k$.
We characterize the family $\mc I$ (resp. $\mc F$, $\mc F'$) of rational functions $u\in k$ such that
any separable quadratic extension $K$ of $k$ in which the infinite prime $\infty = (1/T)$ of $k$ ramifies (resp. splits, is inert)
can be written as $K = k(x_{u})$ with a unique $u\in\mc I$ (resp. $u\in\mc F$, $u\in\mc F'$).
For almost all $s\in\mb C$ with ${\rm Re}(s)\ge \frac{1}2$, we obtain the asymptotic formulas for the summation of $L(s,\chi_{u})$ over all $k(x_{u})$ with $u\in \mc I$,
all $k(x_{u})$ with $u\in \mc F$ or all $k(x_{u})$ with $u\in \mc F'$ of given genus.
As applications, we obtain the asymptotic mean value formulas of $L$-functions at $s=\frac{1}2$ and $s=1$
and the asymptotic mean value formulas of the class number $h_{u}$ or the class number times regulator $h_{u} R_{u}$.
\end{abstract}
\maketitle
\section{Introduction}
In Disquisitiones Arithmeticae \cite{Gauss}, Gauss presented two famous conjectures
concerning the average values of class numbers associated with binary quadratic forms over $\mb Z$,
which can be restated as the average values of class numbers of orders in quadratic number fields.
The imaginary case of these conjecture was first proved by Lipschitz, and the real case by Siegel~\cite{Si44}.
By the Dirichlet's class number formula, these two conjectures can be stated as averages of the values of
quadratic Dirichlet $L$-functions at $s=1$.
In~\cite{Si44}, Siegel showed how to average over all discriminants.
Let $\chi_{d}$ be the quadratic character defined by the Kronecker symbol $\chi_{d}(n) = (\frac{d}{n})$ and
$$
L(s,\chi_{d}) = \sum_{n=1}^{\infty} \chi_{d}(n) n^{-s}
$$
be the Dirichlet series associated to $\chi_{d}$.
Siegel \cite{Si44} has obtained averaging formulas for $L(1,\chi_{d})$ over all positive discriminants $d$ between $1$ and $N$,
or all negative discriminants $d$ such that $1<|d|\le N$.
From these averaging formulas with Dirichlet's class number formula,
Siegel has obtained averaging formulas for the class number $h_{d}$
or the class number times regulator $h_{d} R_{d}$ over all positive discriminants $d$ between $1$ and $N$,
or all negative discriminants $d$ such that $1<|d|\le N$.
At the critical point $s=\frac{1}2$, Jutila~\cite{Ju81} derived an asymptotic formula for
$$
\sum_{d} L^{k}(\tfrac{1}2, \chi_{d}) ~~(k=1,2),
$$
where $d$ runs over fundamental discriminants in the interval $0<d\le X$.
In \cite{GH85}, using the Eisenstein series of $\frac{1}2$-integral weight, for $s\in\mb C$ with ${\rm Re}(s)\ge 1$,
Goldfeld and Hoffstein obtained an asymptotic formula for $\sum L(s, \chi_{m})$,
where the sum is either over positive square-free $m$ between $1$ and $N$, or over negative square-free $m$ with $1<|m|\le N$.
Putting $s=1$ and using the Dirichlet's class number formula, one can average the class number $h_{m}$
or the class number times regulator $h_{m} R_{m}$ over the fields $\mb Q(\sqrt{m})$.

Now, we introduce the analogous results in function fields over finite fields.
Let $k = \bF_{q}(T)$ be the rational function field over a finite field $\bF_{q}$ of $q$ elements,
and $\A = \bF_{q}[T]$ be the ring of polynomials.
First, we consider the case of $q$ being odd.
Let $\chi_{N}$ be the quadratic character defined by the Kronecker symbol $\chi_{N}(f) = (\frac{N}{f})$ and
$$
L(s,\chi_{N}) = \sum_{\substack{f\in\A\\ f : \text{monic}}} \chi_{N}(f) |f|^{-s}
$$
be the quadratic Dirichlet $L$-function associated to $\chi_{N}$.
In \cite{HR92}, Hoffstein and Rosen has obtained a result on averaging $L(1,\chi_{N})$
over all non-square monic polynomials $N\in\A$ (all discriminants) of given degree.
Using the averaging of $L(1,\chi_{N})$ with the Dirichlet's class number formula, they solved the problem of
averaging the class number $h_{N}$ or the class number times regulator $h_{N} R_{N}$
over all non-square monic polynomials $N\in\A$ of given degree.
Moreover, by defining and analyzing metaplectic Eisenstein series in function fields,
Hoffstein and Rosen~\cite{HR92} also succeed in averaging $L(s,\chi_{N})$
over all square-free polynomials $N$ (fundamental discriminants) of given degree.
Andrade and Keating~\cite{AK12} and Jung~\cite{Ju13} have obtained asymptotic formulas for $\sum L(\tfrac{1}2,\chi_{N})$,
where the sum is over all monic square-free $N$ of given degree,
by elementary analytic method using approximate functional equation.
Recently, Andrade~\cite{An12} and Jung~\cite{Ju14} also have obtained asymptotic formulas for $\sum L(1,\chi_{N})$.

In the case of $q$ being even, Chen~\cite{Ch08} obtained formulas of average values of $L$-functions associated to
orders in quadratic function fields, and then derived formulas of average class numbers of these orders.
The aim of this paper is to solve the problem of averaging $L(s, \chi_{u})$ over fundamental discriminants in even characteristic.
Any separable quadratic extension $K$ of $k$ is of the form $K = k(x_{u})$,
where $x_{u}$ is a zero of $X^2+X+u=0$ for some $u\in k$.
We characterize three families $\mc F$, $\mc F'$ and $\mc I$ of rational functions $u\in k$
such that any separable quadratic extension $K$ of $k$ can be written uniquely as $K = K(x_{u})$,
where $u\in\mc F$, $u\in\mc F'$ or $u\in\mc I$ according as the infinite prime $\infty = (1/T)$ of $k$
splits, is inert or is ramified in $K$ (see \S2.2).
By extending the analytic methods in \cite{AK12, An12, AK13, AK14, Ju13, Ju14, RW15}, for almost all $s\in\mb C$ with ${\rm Re}(s)\ge \frac{1}2$,
we obtain the asymptotic formulas for the summation of $L(s,\chi_{u})$ over all $K_{u}$ with $u\in \mc F$,
all $K_{u}$ with $u\in \mc F'$ or all $K_{u}$ with $u\in \mc I$ of given genus (see Theorem \ref{mean-1}).
In \cite{AK12, An12, Ju13, Ju14}, the authors have obtained the asymptotic formulas for the summation of $L(s,\chi_{N})$ only at $s=\frac{1}2$ and $s=1$.
The asymptotic formulas for the summation of $L(s,\chi_{u})$ obtained in this paper hold for almost all $s\in\mb C$ with ${\rm Re}(s)\ge \frac{1}2$.
This can be regarded as an analogue of the result of Hoffstein and Rosen~\cite{HR92} in even characteristic.
The method of proving ``Approximate" functional equations of $L$-functions (Lemma \ref{E-A-FE}) in this paper also holds in odd characteristic case
and all results in \S \ref{S3-2} hold in odd characteristic too.
Thus our method can also be applied to obtain the asymptotic formulas for the summation of $L(s,\chi_{N})$
for almost all $s\in\mb C$ with ${\rm Re}(s)\ge \frac{1}2$ in odd characteristic.
The methods and calculations in \cite{An12, AK12, HR92} also can give the same asymptotic formulas for ${\rm Re}(s)\ge \frac{1}2$.
As applications, we obtain the asymptotic mean value formulas of $L$-functions at $s=\frac{1}2$ and $s=1$
(see Corollaries \ref{mean-1/2-corollary} and \ref{mean-1-corollary}),
and using the Dirichlet's class number formula, we also obtain the asymptotic mean value formulas
of the class number $h_{u}$ or the class number times regulator $h_{u} R_{u}$
(see Theorem \ref{mean-class-number} and Corollary \ref{mean-class-number-coro}).

\section{Statement of results}
\subsection{Some Background on $\A = \mathbb{F}_{q}[T]$}
Let $q$ be a power of $2$.
Let $k := \bF_{q}(T)$ be the rational function field with a constant field $\bF_{q}$,
$\infty = (1/T)$ the infinite prime of $k$, and $\A := \bF_{q}[T]$.
We denote by $\A^{+}$ the set of monic polynomials in $\A$
and by $\mb P$ the set of monic irreducible polynomials in $\A$.
Throughout this paper, any monic irreducible polynomial $P\in\mb{P}$ will be also called a ``prime" polynomial.
For a positive integer $n$ we denote by $\A_{n}^{+}$ the set of monic polynomials in $\A$ of degree $n$
and by $\mb P_{n}$ the set of monic irreducible polynomials in $\A$ of degree $n$.
The \textit{norm} $|f|$ of a polynomial $f\in\A$ is defined to be $|f|:= \#(\A/f\A) =q^{\deg(f)}$ for $f\neq0$,
and $|f|:=0$ for $f=0$, where $\# X$ denotes the cardinality of a set $X$.
For any $0 \ne f \in \A$, let $\Phi(f) := \#(\A/f\A)^{\times}$, and let $sgn(f)$ be the leading coefficient of $f$.
Let $\wp : k \to k$ be the additive homomorphism defined by $\wp(x) = x^2+x$.

The zeta function $\zeta_{\A}(s)$ of $\A$ is defined for $\mathrm{Re}(s)>1$ to be the following infinite series:
\begin{equation}\label{eq:zetaA}
\zeta_{\A}(s) := \sum_{f\in \A^{+}} \frac{1}{|f|^{s}}
= \prod_{P\in\mb P}\left(1-\frac{1}{|P|^{s}}\right)^{-1}, \quad \mathrm{Re}(s)>1.
\end{equation}
It is well known that $\zeta_{\A}(s)=\frac{1}{1-q^{1-s}}$.

The monic irreducible polynomials in $\A$ also satisfies the analogue of the Prime Number Theorem.
In other words we have the following.
\begin{thm}[Prime Polynomial Theorem] \label{thm:pnt}
Let $\pi_{\A}(n)$ denote the number of monic irreducible polynomials in $\A$ of degree $n$.
Then, we have
\begin{equation}
\pi_{\A}(n)=\#\mb{P}_{n}=\frac{q^{n}}{n}+O\left(\frac{q^{\tfrac{n}{2}}}{n}\right).
\end{equation}
\end{thm}

\subsection{Quadratic function field in even characteristic}
Any separable quadratic extension $K$ of $k$ is of the form $K = K_{u} := k(x_{u})$,
where $x_{u}$ is a zero of $X^2+X+u=0$ for some $u\in k$.
We say that $u, v \in k$ are equivalent if $K_{u} = K_{v}$.
It is known that $u$ and $v$ are equivalent if and only if $u+v = \wp(w)$ for some $w\in k$ (see \cite{Ha34} or \cite{HL10}).
Fix an element $\xi \in \bF_{q}\setminus \wp(\bF_{q})$.
The following lemma is due to Y. Li.
In fact, Y. Li obtained the result which holds for any Artin-Schreier extensions of the rational function fields of any characteristic.

\begin{lem}\label{normalization}
Any separable quadratic extension $K$ of $k$ is of the form $K=K_{u}$,
where $u\in k$ can be uniquely normalized to satisfy the following conditions:
\begin{align}\label{normalization-2}
u = \sum_{i=1}^{m} \sum_{j=1}^{e_{i}} \frac{Q_{ij}}{P_{i}^{2j-1}} + \sum_{\ell=1}^{n} \alpha_{\ell} T^{2\ell-1} + \alpha,
\end{align}
where $P_{i}\in\mb P$ are distinct, $Q_{ij}\in\A$ with $\deg(Q_{ij}) < \deg(P_{i})$, $Q_{i e_{i}} \ne 0$, $\alpha \in \{0, \xi\}$, $\alpha_{\ell} \in\bF_{q}$
and $\alpha_{n} \ne 0$ for $n>0$.
\end{lem}
\begin{proof}
Since it is difficult to find the reference for it, we will give the proof due to Y. Li.
We know that every element $u\in k$ can be uniquely written as
$$
u(T) = \sum_{i=1}^{m} \sum_{e_{ij}=1}^{e_{i}} \frac{Q_{ij}}{P_{i}^{e_{ij}}} + \sum_{\ell=1}^{n} \alpha_{\ell} T^{\ell} + \alpha,
$$
where $P_{i} \in\mb P$ are distinct, $\deg(Q_{ij}) < \deg(P_{i})$, $Q_{i e_{i}} \ne 0$ and $\alpha_{\ell}, \alpha \in \bF_{q}$.
We can remove the term $\frac{Q_{ij}}{P_{i}^{e_{ij}}}$ with $2|e_{ij}$ as follows.
Let $e_{ij} = 2k_{ij}$ and let $M_{ij}\in\A$ with $\deg(M_{ij}) < \deg(P_{i})$ such that
$$
M_{ij}^2 \equiv Q_{ij} \bmod P_{i}.
$$
Then we have
$$
\frac{Q_{ij}}{P_{i}^{e_{ij}}} + \left(\frac{M_{ij}}{P_{i}^{k_{ij}}}\right)^2 + \frac{M_{ij}}{P_{i}^{k_{ij}}}
= \frac{Q_{ij} + M_{ij}^{2}}{P_{i}^{e_{ij}}} + \frac{M_{ij}}{P_{i}^{k_{ij}}}.
$$
Similarly, we can remove even degree term $\alpha_{2\ell} T^{2\ell}$ and we get the desired form.
Now it is clear that any two $u, v \in k$ which are normalized as in \eqref{normalization-2} are equivalent if and only if $u=v$.
\end{proof}

Let $u\in k$ be normalized one as in \eqref{normalization-2}.
The infinite prime $\infty = (1/T)$ splits, is inert or ramified in $K_{u}$ according as
$n=0$ and $\alpha=0$, $n=0$ and $\alpha = \xi$, or $n>0$.
Then the field $K_{u}$ is called \emph{real}, \emph{inert imaginary}, or \emph{ramified imaginary}, respectively.
The discriminant $D_{u}$ of $K_{u}$ is given by
$$
D_{u} = \prod_{i=1}^{m} P_{i}^{2e_i} \quad \text{ if }~~ n = 0
$$
and
$$
D_{u} = \prod_{i=1}^{m} P_{i}^{2e_i} \cdot (1/T)^{2n} \quad \text{ if ~~$n>0$,}
$$
and, by the Hurwitz genus formula (\cite[Theorem III.4.12]{St93}), the genus $g_{u}$ of $K_{u}$ is given by
\begin{align}\label{genus formula}
g_{u} = \frac{1}{2} \deg(D_{u}) - 1.
\end{align}

For $M\in\A^{+}$, let $r(M) := \prod_{P|M} P$ and $t(M) := M \cdot r(M)$.
For $P\in\PP$, let $\nu_{P}$ be the normalized valuation at $P$, that is, $\nu_{P}(M) = e$, where $P^{e}\| M$.
Let $\mc B$ be the set of monic polynomials $M$ such that $\nu_{P}(M) = 0$ or odd for any $P\in\mb P$,
that is, $t(M)$ is a square, and $\mc C$ be the set of rational functions $\frac{D}{M} \in k$ such that
$D\in\A, M \in\mc B$ and $\deg(D) < \deg(M)$.
For $M\in\mc B$, let $\ell_{P} := \frac{1}{2}(\nu_{P}(M)+1)$  for any $P|M$.
Also we let $\mc E$ be the set of rational functions $\frac{D}{M} \in \mc C$ of the form
$$
\frac{D}{M} = \sum_{P|M} \sum_{i=1}^{\ell_{P}} \frac{A_{P,i}}{P^{2i-1}},
$$
where $\deg(A_{P,i}) < \deg(P)$ for any $P|M$ and for all $1 \le i \le \ell_{P}$.
Note that for $\frac{D}{M}\in\mc E$, $\gcd(D,M) = 1$ if and only if $A_{P,\ell_{P}} \ne 0$ for all $P|M$.
Let $\mc F$ be the set of rational functions $\frac{D}{M} \in \mc E$ such that $A_{P,\ell_{P}} \ne 0$ for all $P|M$
and $\mc F' := \{u+\xi : u \in \mc F\}$.
By the normalization in \eqref{normalization-2}, we can see that $u\mapsto K_{u}$ defines a one-to-one correspondence
between $\mc F$ (resp. $\mc F'$) and the set of all real (resp. inert imaginary) separable quadratic extensions of $k$.
For any positive integer $s$, let $\mc G_{s}$ be the set of polynomials $F(T) \in \A$ of the form
$$
F(T) = \alpha + \sum_{i=1}^{s} \alpha_{i} T^{2i-1},
\quad\text{ where $\alpha \in \{0, \xi\}, \alpha_{i}\in\mb F_{q}$ and $\alpha_{s} \ne 0$.}
$$
Let $\mc G := \bigcup_{s\ge 1} \mc G_{s}$ and $\mc I := \{u+F : u \in \bar{\mc F} \text{ and } F \in \mc G\}$,
where $\bar{\mc F} = \mc F \cup \{0\}$.
By the normalization in \eqref{normalization-2}, we can see that $w\mapsto K_{w}$
defines a one-to-one correspondence between $\mc I$
and the set of all ramified imaginary separable quadratic extensions of $k$.

\subsection{Hasse symbol and $L$-functions}
Let $P\in\mb P$.
For any $u\in k$ whose denominator is not divisible by $P$, the Hasse symbol $[u, P)$ with values in $\bF_{2}$ is defined by
$$
[u, P) := \begin{cases}
0 & \text{ if $X^2+X\equiv u ~(\bmod P)$ is solvable in $\A$,} \\ 1 & \text{ otherwise.}
\end{cases}
$$
For $N\in\A$ prime to the denominator of $u$, write $N = sgn(N) \prod_{i=1}^{s} P_{i}^{e_i}$, where $P_{i} \in \mb P$ are distinct and $e_{i} \ge 1$,
and define $[u, N)$ to be $\sum_{i=1}^{s} e_{i} [u, P_{i})$.

For $u\in k$ and $0\ne N \in\A$, we also define the quadratic symbol:
$$
\left\{\frac{u}{N}\right\} := \begin{cases}
(-1)^{[u, N)} & \text{ if $N$ is prime to the denominator of $u$}, \\ 0 & \text{ otherwise.}
\end{cases}
$$
This symbol is clearly additive in its first variable, and multiplicative in the second variable.

For the field $K_{u}$, we associate a character $\chi_{u}$ on $\A^{+}$ which is defined by $\chi_{u}(f) = \{\frac{u}{f}\}$,
and let $L(s,\chi_{u})$ be the $L$-function  associated to the character $\chi_{u}$: for $s\in\mb C$ with ${\rm Re}(s)\ge 1$,
$$
L(s,\chi_{u}) := \sum_{f\in\A^{+}} \frac{\chi_{u}(f)}{|f|^{s}} = \prod_{P\in\mb P} \left(1-\frac{\chi_{u}(P)}{|P|^s}\right)^{-1}.
$$
It is well known that $L(s,\chi_{u})$ is a polynomial in $q^{-s}$.
Letting $z=q^{-s}$, write $\mc L(z,\chi_{u}) = L(s, \chi_{u})$.
Then, $\mc L(z,\chi_{u})$ is a polynomial in $z$ of degree $2 g_{u} + \frac{1}{2}(1+(-1)^{\varepsilon(u)})$,
where $\varepsilon(u) = 1$ if $K_{u}$ is ramified imaginary and $\varepsilon(u) = 0$ otherwise.
Also we have that $\mathcal{L}(z,\chi_{u})$ has a ``trivial" zero at $z=1$ (resp. $z=-1$)
if and only if $K_{u}$ is real (resp. inert imaginary), so we can define the ``completed" $L$--function as
\begin{equation}\label{completed-L}
\mc{L}^{*}(z,\chi_{u}) := \begin{cases}
\mc{L}(z,\chi_{u}) & \text{ if $K_{u}$ is ramified imaginary,} \\
(1-z)^{-1} \mc{L}(z,\chi_{u}) & \text{ if $K_{u}$ is real,} \\
(1+z)^{-1} \mc{L}(z,\chi_{u}) & \text{ if $K_{u}$ is inert imaginary,}
\end{cases}
\end{equation}
which is a polynomial of even degree $2 g_{u}$ satisfying the functional equation
\begin{equation}\label{functional equation}
\mathcal{L}^{*}(z,\chi_{u})=(q z^{2})^{g_{u}}\mathcal{L}^{*}\left(\frac{1}{qz},\chi_{u}\right).
\end{equation}

\subsection{Main results}
For $M\in\mc B$, let
$$
\tilde{M} := \prod_{P|M} P^{(\nu_{P}(M)+1)/2} = \sqrt{t(M)}.
$$
For positive integer $n$, let
\begin{align*}
&\mc B_{n} := \left\{M \in \mc B : \deg(t(M)) = 2n\right\},
\hspace{1em}\mc C_{n} := \left\{\tfrac{D}{M} \in \mc C : M\in\mc B_{n}\right\}, \\
&\mc E_{n} := \mc E \cap \mc C_{n}, \hspace{1.4em} \mc F_{n} := \mc F \cap \mc E_{n},
\hspace{1.4em} \mc F'_{n} := \left\{u+\xi : u \in \mc F_{n}\right\}.
\end{align*}
Under the above correspondence $u\mapsto K_{u}$, $\mc F_{n}$ (resp. $\mc F'_{n}$) corresponds to
the set of all real (resp. inert imaginary) separable quadratic extensions $K_{u}$ of $k$ with genus $n-1$.

Let $\mc F_{0} = \{0\}$.
For any integers $r\ge 0$ and $s \ge 1$, let $\mc I_{(r,s)} = \{u+F : u\in\mc F_{r}, F\in\mc G_{s}\}$.
For any integer $n \ge 1$, let $\mc I_{n}$ be the union of all $\mc I_{(r,s)}$,
where $(r,s)$ runs over all pairs of nonnegative integers such that $s>0$ and $r+s = n$.
Then, under the correspondence $u\mapsto K_{u}$, $\mc I_{n}$ corresponds to
the set of all ramified imaginary separable quadratic extensions $K_{u}$ of $k$ with genus $n-1$.

\begin{lem}
For a positive integer $n$, we have
$\#\mc B_{n} = q^{n}$, $\#\mc E_{n} = q^{2n}$, $\#\mc F_{n} = \zeta_{\A}(2)^{-1} q^{2n}$
and $\#\mc I_{n} = 2 \zeta_{\A}(2)^{-1} q^{2n-1}$.
\end{lem}
\begin{proof}
The map $\mc B_{n}\to \A_{n}^{+}$ defined by $M\mapsto \tilde{M}$ and the map $\A_{n}^{+}\to \mc B_{n}$
defined by $N\mapsto N^{*} := {N^2}/{r(N)}$ are inverse to each other.
Thus we have $\#\mc B_{n} = \#\A_{n}^{+} = q^{n}$.
For each $M\in\mc B_{n}$, there are $q^{n}$ (resp. $\Phi(\tilde{M})$) $D$'s
such that $\frac{D}{M}\in\mc E_{n}$ (resp. $\frac{D}{M}\in\mc F_{n}$).
Thus $\#\mc E_{n} = q^{n} \cdot \#\mc B_{n} = q^{2n}$ and
$$
\#\mc F_{n} = \sum_{M\in\mc B_{n}} \Phi(\tilde{M})
= \sum_{\tilde{M}\in\A_{n}^{+}} \Phi(\tilde{M}) = \zeta_{\A}(2)^{-1} q^{2n}
$$
by \cite[Proposition 2.7]{Ro02}.
Since $\#\mc G_{s} = 2 \zeta_{\A}(2)^{-1} q^{s}$ for $s \ge 1$, we have
\begin{align*}
\#\mc I_{n} = \sum_{s=1}^{n} \#\mc I_{(n-s,s)}
= \sum_{s=1}^{n} \#\mc F_{n-s} \cdot \#\mc G_{s} = 2 \zeta_{\A}(2)^{-1} q^{2n-1}.
\end{align*}
\end{proof}

For any $s\in\mb C$ with ${\rm Re}(s)>0$, let
$$
{\rm P}(s) = \prod_{P}\left(1-\frac{1}{|P|^{s}(|P|+1)}\right).
$$
For an arbitrary small $\varepsilon>0$, let ${X}_{\varepsilon}$ be the set of complex numbers $s \in \mb C$
such that $\tfrac{1}2 \le {\rm Re}(s) < 1$ and $|s-\tfrac{1}2|>\varepsilon$, and
$\bar{X}_{\varepsilon} = \{\tfrac{1}2\} \cup {X}_{\varepsilon} \cup \{s\in \mb C : {\rm Re}(s) \ge 1\}$.
For the first moment of Dirichlet $L$--functions at $s\in\mb C$ with ${\rm Re}(s) \ge \frac{1}2$, we have the following theorem.

\begin{thm}\label{mean-1}
Let $s_{1} = \frac{1}2(1+\log_{q}2) \le 1$.
\begin{enumerate}
\item
For an arbitrary small $\varepsilon>0$ and $s\in \bar{X}_{\varepsilon}$, we have
\begin{align*}
\sum_{u\in\mc I_{g+1}} L(s,\chi_{u}) = 2 \alpha_{g}(s) q^{2g+1} + \begin{cases}
O(g 2^{\frac{g}2} q^{(2-s)g}) & \text{ if ${\rm Re}(s) < s_{1}$,} \\
O(g q^{\frac{3g}2}) & \text{ if ${\rm Re}(s) \ge s_{1}$,}
\end{cases}
\end{align*}
where $\alpha_{g}(\frac{1}2) = \frac{{\rm P}(1)}{\zeta_{\A}(2)}\big(g+1+\frac{2}{\log q} \frac{{\rm P}'}{\rm P}(1)\big)$,
$\alpha_{g}(s) = \frac{\zeta_{\A}(2s)}{\zeta_{\A}(2)} {\rm P}(2s)$ for ${\rm Re}(s) \ge 1$, and, for $s\in X_{\varepsilon}$,
\begin{align*}
\alpha_{g}(s) &= \frac{\zeta_{\A}(2s)}{\zeta_{\A}(2)} \left\{{\rm P}(2s) - q^{(1-2s)(g+1)} {\rm P}(2-2s) \right. \\
&\hspace{5em} \left. + {\rm P}(1)\left(q^{(1-2s)(g-[\frac{g-1}2])} - q^{(1-2s)([\frac{g}2]+1)}\right)\right\}.
\end{align*}

\item
For an arbitrary small $\varepsilon>0$ and $\delta>0$
and for $s \in \bar{X}_{\varepsilon}$ with $|s-1|>\delta$ or $s=1$, we have
\begin{align}
\sum_{u\in\mc F_{g+1}}L(s,\chi_{u})
&= \beta_{g}(s) q^{2g+2} + \begin{cases}
O(2^{\frac{g}2} q^{(2-s)g}) & \text{ if ${\rm Re}(s) < s_{1}$,} \\
O(g q^{\frac{3g}2}) & \text{ if ${\rm Re}(s) \ge s_{1}$,}
\end{cases}
\end{align}
where $\beta_{g}(\frac{1}2) = \frac{{\rm P}(1)}{\zeta_{\A}(2)}\big(g+1+\zeta_{\A}(\frac{1}2)+\frac{2}{\log q} \frac{{\rm P}'}{\rm P}(1)\big)$,
$\beta_{g}(s) = \frac{\zeta_{\A}(2s)}{\zeta_{\A}(2)} {\rm P}(2s)$ for ${\rm Re(s)} \ge 1$, and, for ${\rm Re}(s) < 1 \,(s\ne\frac{1}2)$,
\begin{align*}
\hspace{3em}\beta_{g}(s) &= \frac{\zeta_{\A}(2s)}{\zeta_{\A}(2)} \left\{{\rm P}(2s) - {\rm P}(1) q^{(1-2s)([\frac{g}2]+1)} \right. \\
&\hspace{6em} \left. + \tfrac{\zeta_{\A}(2-s)}{\zeta_{\A}(1+s)} q^{(1-2s)(g+1)} \left({\rm P}(1) q^{(2s-1)([\frac{g-1}2]+1)} - {\rm P}(2-2s)\right)\right\}  \\
&\hspace{5em} - {\rm P}(1) q^{-gs} \left(q^{[\frac{g}2]-s} + \tfrac{\zeta_{\A}(2-s)}{\zeta_{\A}(1+s)} q^{[\frac{g-1}2]}\right).
\end{align*}

\item
For an arbitrary small $\varepsilon>0$ and $s\in \bar{X}_{\varepsilon}$, we have
\begin{align*}
\sum_{u\in\mc F'_{g+1}}L(s,\chi_{u}) = \gamma_{g}(s) q^{2g+2} + \begin{cases}
O(2^{\frac{g}2} q^{(2-s)g}) & \text{ if ${\rm Re}(s) < s_{1}$,} \\
O(g q^{\frac{3g}2}) & \text{ if ${\rm Re}(s) \ge s_{1}$,}
\end{cases}
\end{align*}
where $\gamma_{g}(\frac{1}2) = \frac{{\rm P}(1)}{\zeta_{\A}(2)}\big(g+1+\frac{\zeta_{\A}(0)}{\zeta_{\A}(\frac{1}2)}+\frac{2}{\log q} \frac{{\rm P}'}{\rm P}(1)\big)$,
$\gamma_{g}(s) = \frac{\zeta_{\A}(2s)}{\zeta_{\A}(2)} {\rm P}(2s)$ for ${\rm Re(s)} \ge 1$, and, for $s\in X_{\varepsilon}$,
\begin{align*}
\hspace{3em}\gamma_{g}(s) &= \frac{\zeta_{\A}(2s)}{\zeta_{\A}(2)} \left\{{\rm P}(2s) - {\rm P}(1) q^{(1-2s)([\frac{g}2]+1)} \right. \\
&\hspace{6em} \left. + \left(\tfrac{1+q^{-s}}{1+q^{s-1}}\right) q^{(1-2s)(g+1)} \left({\rm P}(1) q^{(2s-1)([\frac{g-1}2]+1)} - {\rm P}(2-2s)\right)\right\}  \\
&\hspace{5em} + (-1)^{g} {\rm P}(1) q^{-gs} \left(q^{[\frac{g}2]-s} - \left(\tfrac{1+q^{-s}}{1+q^{s-1}}\right) q^{[\frac{g-1}2]}\right).
\end{align*}
\end{enumerate}
\end{thm}

\begin{rem}
The restrictions of $|s-\frac{1}2|>\varepsilon$ in $X_{\varepsilon}$ and $|s-1|>\delta$ in Theorem \ref{mean-1}
are caused by the facts that $\zeta_{\A}(2s)$ and $\zeta_{\A}(2-s)$ are unbounded as $s\to \frac{1}2$ and $s\to 1$, respectively.
\end{rem}

Since $\#\mc I_{g+1} = 2 \zeta_{\A}(2)^{-1} q^{2g+1}$ and
$\#\mc F_{g+1} = \#\mc F'_{g+1} = \zeta_{\A}(2)^{-1} q^{2g+2}$,
we get from Theorem \ref{mean-1} the following asymptotic mean value formulas of $L$-functions at $s=\frac{1}2$ and $s=1$.
Compare with the results in \cite{AK12, Ju13, An12, Ju14}.

\begin{cor}\label{mean-1/2-corollary}
Assume that $q>2$ is fixed.
As $g\to \infty$, we have
\begin{enumerate}
\item
$\dis \frac{1}{\#\mc I_{g+1}}\sum_{u\in\mc I_{g+1}} L(\tfrac{1}2,\chi_{u}) \sim (g+1) {\rm P}(1)$,
\item
$\dis \frac{1}{\#\mc F_{g+1}} \sum_{u\in\mc F_{g+1}}L(\tfrac{1}2,\chi_{u}) \sim (g+1) {\rm P}(1)$,
\item
$\dis \frac{1}{\#\mc F'_{g+1}}\sum_{u\in\mc F'_{g+1}}L(\tfrac{1}2,\chi_{u}) \sim (g+1) {\rm P}(1)$.
\end{enumerate}
\end{cor}

\begin{cor}\label{mean-1-corollary}
Assume that $q>2$ is fixed.
As $g\to \infty$, we have
\begin{enumerate}
\item
$\dis\frac{1}{\#\mc I_{g+1}} \sum_{u\in\mc I_{g+1}} L(1, \chi_{u}) \sim \zeta_{\A}(2) {\rm P}(2)$,
\item
$\dis\frac{1}{\#\mc F_{g+1}}\sum_{u\in\mc F_{g+1}} L(1, \chi_{u}) \sim \zeta_{\A}(2) {\rm P}(2)$,
\item
$\dis\frac{1}{\#\mc F'_{g+1}}\sum_{u\in\mc F'_{g+1}} L(1, \chi_{u}) \sim \zeta_{\A}(2) {\rm P}(2)$.
\end{enumerate}
\end{cor}

Let $\OO_{u}$ be the integral closure of $\A$ in the quadratic function field $K_{u}$
and $h_{u}$ be the ideal class number of $\OO_{u}$.
If $K_{u}$ is real, $R_{u}$ denotes the regulator of $\OO_{u}$.
We have the following formula which connects $L(1, \chi_{u})$ with $h_{u}$ (\cite[Theorem 5.2]{CY08}):
\begin{align}\label{class-number-formula}
L(1, \chi_{u}) = \begin{cases}
q^{- g_{u}} h_{u}  & \text{ if $K_{u}$ is ramified imaginary,} \\
\zeta_{\A}(2)^{-1} q^{-g_{u}} h_{u} R_{u} & \text{ if $K_{u}$ is real,} \\
\frac{1}{2} \zeta_{\A}(2) \zeta_{\A}(3)^{-1} q^{-g_{u}} h_{u} & \text{ if $K_{u}$ is inert imaginary.} \\
\end{cases}
\end{align}

From Theorem \ref{mean-1} and \eqref{class-number-formula}, we obtain the following theorem.

\begin{thm}\label{mean-class-number}
As $q$ is fixed and $g\to \infty$, we have
\begin{enumerate}
\item
$\dis \sum_{u\in \mc I_{g+1}} h_{u} = 2 {\rm P}(2) q^{3g+1} + O\left(g q^{\frac{5}{2}g}\right)$,
\item
$\dis \sum_{u\in\mc F_{g+1}} h_{u} R_{u} = \zeta_{\A}(2) {\rm P}(2) q^{3g+2} + O\left(g q^{\frac{5}{2}g}\right)$,
\item
$\dis \sum_{u\in\mc F'_{g+1}} h_{u} = 2 \zeta_{\A}(2)^{-1} \zeta_{\A}(3) {\rm P}(2) q^{3g+2} + O\left(g q^{\frac{5}{2}g}\right)$.
\end{enumerate}
\end{thm}

\begin{cor}\label{mean-class-number-coro}
As $q$ is fixed and $g\to \infty$, we have
\begin{enumerate}
\item
$\dis \frac{1}{\#\mc I_{g+1}} \sum_{u\in\mc I_{g+1}} h_{u} \sim \zeta_{\A}(2) {\rm P}(2) q^{g}$,
\item
$\dis \frac{1}{\#\mc F_{g+1}}\sum_{u\in\mc F_{g+1}} h_{u} R_{u} \sim \zeta_{\A}(2)^2 {\rm P}(2) q^{g}$,
\item
$\dis \frac{1}{\#\mc F'_{g+1}}\sum_{u\in\mc F'_{g+1}} h_{u} \sim 2 \zeta_{\A}(3) {\rm P}(2) q^{g}$.
\end{enumerate}
\end{cor}

\section{Preliminaries}

\subsection{``Approximate" functional equations of $L$--functions}\label{S3-1}
Let $u \in k$ be a normalized one as in \eqref{normalization-2} and write
\begin{align*}
L(s,\chi_{u}) = \sum_{n=0}^{d_{u}} A_{u}(n) q^{-sn} \quad\text{ with } A_{u}(n) = \sum_{f\in\A_{n}^{+}} \chi_{u}(f),
\end{align*}
where $d_{u} = 2 g_{u} + \frac{1}{2}(1+(-1)^{\varepsilon(u)})$.
In this subsection we prove the following lemma, which is a generalization of Lemmas 2.1 in \cite[]{An12, Ju13, Ju14}.
We remark that the proof of Lemma \ref{E-A-FE} also can be applied to obtain ``Approximate" functional equations
of $L$--functions $L(s,\chi_{N})$ in odd characteristic.

\begin{lem}\label{E-A-FE}
Let $s\in\mb C$ with ${\rm Re}(s)\ge\frac{1}2$.
\begin{enumerate}
\item
For $u\in\mc I$, we have
\begin{align}\label{Approximate-f-eq-ramified}
L(s,\chi_{u}) = \sum_{n=0}^{g_{u}} A_{u}(n) q^{-sn} + q^{(1-2s)g_{u}} \sum_{n=0}^{g_{u}-1} A_{u}(n) q^{(s-1)n}.
\end{align}
\item
For $u\in\mc F$, we have
\begin{align}\label{E-A-FE-Real}
L(s,\chi_{u}) = \sum_{n=0}^{g_{u}} A_{u}(n) q^{-sn} - q^{-(g_{u}+1)s} \sum_{n=0}^{g_{u}} A_{u}(n) + H_{u}(s),
\end{align}
where $H_{u}(1) := \zeta_{\A}(2)^{-1} q^{-g_{u}}\dis\sum_{n=0}^{g_{u}-1} \left(g_{u}-n\right) A_{u}(n)$ and, for $s\ne 1$,
$$
H_{u}(s) := q^{(1-2s)g_{u}} \eta(s) \sum_{n=0}^{g_{u}-1} q^{(s-1)n} A_{u}(n)
- q^{-s g_{u}} \eta(s) \sum_{n=0}^{g_{u}-1} A_{u}(n)
$$
with $\eta(s) = \frac{\zeta_{\A}(2-s)}{\zeta_{\A}(1+s)}$.
\item
For $u\in\mc F'$, we have
\begin{align}\label{A-FE-Inert}
\hspace{2.5em}L(s,\chi_{u}) &= \sum_{n=0}^{g_{u}} A_{u}(n) q^{-sn} + q^{-(g_{u}+1)s} \sum_{n=0}^{g_{u}}(-1)^{n+g_{u}} A_{u}(n) \nonumber \\
&\quad + \nu(s) q^{(1-2s)g_{u}} \sum_{n=0}^{g_{u}-1} A_{u}(n) q^{(s-1)n}
+ \nu(s) q^{-sg_{u}}\sum_{n=0}^{g_{u}-1} (-1)^{n+g_{u}+1} A_{u}(n)
\end{align}
with $\nu(s) = \frac{1+q^{-s}}{1+q^{s-1}}$.
\end{enumerate}
\end{lem}
\begin{proof}
Write
\begin{equation*}
\mc L(z,\chi_{u}) = \sum_{n=0}^{d_{u}} A_{u}(n) z^{n} \quad\text{ and }\quad \mc L^{*}(z,\chi_{u}) = \sum_{n=0}^{2g_{u}} A_{u}^{*}(n) z^{n}.
\end{equation*}
By definition \eqref{completed-L}, we have
\begin{align}\label{C-O}
A_{u}^{*}(n) = \begin{cases}
A_{u}(n) & \text{ if $u\in\mc I$,} \\
\sum_{i=0}^{n} A_{u}(i) & \text{ if $u\in\mc F$,} \\
\sum_{i=0}^{n} (-1)^{n-i} A_{u}(i) & \text{ if $u\in\mc F'$.}
\end{cases}
\end{align}
From the functional equation \eqref{functional equation}, we have
\begin{align*}
\sum_{n=0}^{2g_{u}} A_{u}^{*}(n) z^{n} = \sum_{n=0}^{2g_{u}} A_{u}^{*}(n) q^{g_{u}-n} z^{2g_{u}-n}
= \sum_{n=0}^{2g_{u}} A_{u}^{*}(2g_{u}-n) q^{n-g_{u}} z^{n},
\end{align*}
and equating coefficients, we have
\begin{align*}
A_{u}^{*}(n) = A_{u}^{*}(2g_{u}-n) q^{n-g_{u}} \quad\text{ or }\quad A_{u}^{*}(2g_{u}-n) = A_{u}^{*}(n) q^{g_{u}-n}.
\end{align*}
Hence, we can write $\mc L^{*}(z,\chi_{u})$ as
\begin{align}\label{Approximate-FE}
\mc L^{*}(z,\chi_{u}) = \sum_{n=0}^{g_{u}} A_{u}^{*}(n) z^{n}
+ q^{g_{u}} z^{2g_{u}} \sum_{n=0}^{g_{u}-1} A_{u}^{*}(n) q^{-n} z^{-n}.
\end{align}
If $u\in\mc I$, since $\mc L(u,\chi_{u}) = \mc L^{*}(u,\chi_{u})$,
\eqref{Approximate-f-eq-ramified} follows from \eqref{Approximate-FE} immediately by letting $z=q^{-s}$.
Suppose that $u\in\mc F$.
From \eqref{C-O} and \eqref{Approximate-FE}, we have
\begin{align}\label{A-FE-001}
\mc L^{*}(z,\chi_{u}) &= \sum_{n=0}^{g_{u}} \left(\sum_{i=0}^{n} A_{u}(i)\right) z^{n}
+ q^{g_{u}} z^{2 g_{u}} \sum_{n=0}^{g_{u}-1} \left(\sum_{j=0}^{n} A_{u}(j)\right) q^{-n} z^{-n} \nonumber \\
&= \sum_{n=0}^{g_{u}} \left(\frac{z^{n}-z^{g_{u}+1}}{1-z}\right) A_{u}(n) + H^*(z),
\end{align}
where $H^{*}(q^{-1}) := q^{-g_{u}}\dis\sum_{n=0}^{g_{u}-1} \left(g_{u}-n\right) A_{u}(n)$ and, for $z\ne q^{-1}$,
$$
H^{*}(z) := \dfrac{q^{g_{u}} z^{2g_{u}}}{1-q^{-1} z^{-1}} \dis\sum_{n=0}^{g_{u}-1} q^{-n} z^{-n} A_{u}(n)
- \dfrac{z^{g_{u}}}{1-q^{-1} z^{-1}} \dis\sum_{n=0}^{g_{u}-1} A_{u}(n).
$$
Then, by multiplying $(1-z)$ on \eqref{A-FE-001} and putting $z=q^{-s}$, we get \eqref{E-A-FE-Real}.
Finally, consider the case that $u\in\mc F'$.
By \eqref{C-O} and \eqref{Approximate-FE}, we have
\begin{align}\label{A-FE-002}
\mc L^{*}(z,\chi_{u}) &= \sum_{n=0}^{g_{u}} \left(\sum_{i=0}^{n} (-1)^{n-i} A_{u}(i)\right) z^{n}
+ q^{g_{u}} z^{2g_{u}} \sum_{n=0}^{g_{u}-1} \left(\sum_{j=0}^{n} (-1)^{n-j} A_{u}(j)\right) q^{-n} z^{-n} \nonumber \\
&\hspace{-1.5em}= \frac{1}{1+z} \sum_{n=0}^{g_{u}} A_{u}(n) z^{n} + \frac{z^{g_{u}+1}}{1+z} \sum_{n=0}^{g_{u}}(-1)^{n+g_{u}} A_{u}(n) \nonumber\\
&\hspace{-1em}+ \frac{q^{g_{u}} z^{2g_{u}}}{1+q^{-1} z^{-1}} \sum_{n=0}^{g_{u}-1} A_{u}(n) q^{-n} z^{-n}
+ \frac{z^{g_{u}}}{1+q^{-1} z^{-1}} \sum_{n=0}^{g_{u}-1} (-1)^{n+g_{u}+1} A_{u}(n).
\end{align}
By multiplying $(1+z)$ on \eqref{A-FE-002} and putting $z=q^{-s}$, we get \eqref{A-FE-Inert}.
\end{proof}

\subsection{Some auxiliary lemmas}\label{S3-2}
All results in this subsection hold in arbitrary characteristic.
Thus we assume that $q$ is a power of any prime number.

The following lemma is a minor modification of Theorem 17.1 in \cite{Ro02}.

\begin{lem}\label{auxiliary lemmas-1}
Let $\rho : \A^{+} \to \mb C$ and let $\zeta_{\rho}(s)$ be the corresponding Dirichlet series.
Suppose this series converges absolutely in the region $\mathrm{Re}(s)>1$
and is holomorphic in the region $\{s \in B : \mathrm{Re}(s) = 1\}$
except for a simple pole of at $s=1$ with residue $\alpha$, where
$$
B = \left\{s \in \mb C : - \frac{\pi i}{\log q} \le \mathrm{Im}(s) < \frac{\pi i}{\log q} \right\}.
$$
Then, there is a positive real number $\delta < 1$ such that
$$
\sum_{f\in\A_{n}^{+}} \rho(f) = \alpha (\log q) q^{n} + O\left(q^{\delta n}\right).
$$
If $\zeta_{\rho}(s) - \frac{\alpha}{s-1}$ is holomorphic in $\mathrm{Re}(s) \ge \delta'$,
then the error term can be replaced by with $O(q^{\delta' n})$.
\end{lem}

\begin{lem}\label{auxiliary lemmas-2}
Let $L\in\A^{+}$.
Given any $\epsilon>0$, we have
$$
\sum_{\substack{f\in\A_{n}^{+} \\ (f,L)=1}} \Phi(f)
= \frac{1}{\zeta_{\A}(2)} q^{2n} \prod_{P|L} (1+|P|^{-1})^{-1} + O\left(q^{(1+\epsilon) n}\right).
$$
\end{lem}
\begin{proof}
Let $\zeta_{\Phi}(s)$ be the Dirichlet series associated to $\Phi$.
It is known that (\cite[Chap 2. equation (6)]{Ro02})
\begin{align}\label{auxiliary lemmas-2-001}
\zeta_{\Phi}(s) = \frac{\zeta_{\A}(s-1)}{\zeta_{\A}(s)}.
\end{align}
Let $\rho : \A^{+} \to \mb C$ be defined by $\rho(f) = \Phi(f)$ if $(L, f) = 1$ and $\rho(f) = 0$ otherwise,
and $\zeta_{\rho}(s)$ be the Dirichlet series associated to $\rho$:
\begin{align}
\zeta_{\rho}(s) := \sum_{\substack{f\in\A^{+} \\ (f,L)=1}} \Phi(f) |f|^{-s}
= \prod_{\substack{P\in\mb P \\ P\nmid L}}\left(1+\sum_{n=1}^{\infty} \Phi(P^n) |P|^{-ns}\right).
\end{align}
Then, we have
\begin{align*}
\zeta_{\rho}(s) = \zeta_{\Phi}(s) \prod_{\substack{P\in\mb P \\ P|L}}\left(\frac{1-|P|^{1-s}}{1-|P|^{-s}}\right),
\end{align*}
which has a simple pole at $s=2$ and is holomorphic in the region ${\rm Re}(s)>1$.
Hence, $\zeta_{\rho}(s+1)$ has a simple pole at $s=1$ and is holomorphic in the region ${\rm Re}(s)>0$.
Then $\zeta_{\rho}(s+1)$ is holomorphic in the region ${\rm Re}(s) \ge \epsilon$
except for a simple pole at $s=1$ with residue
$$
\alpha = \frac{1}{\zeta_{\A}(2) \log q} \prod_{P|L} (1+|P|^{-1})^{-1}.
$$
Applying Lemma \ref{auxiliary lemmas-1} to $\zeta_{\rho}(s+1)$ with $\delta' = \epsilon$, we find
$$
\sum_{\substack{f\in\A_{n}^{+} \\ (f,L)=1}} \Phi(f)
= \frac{1}{\zeta_{\A}(2)} q^{2n} \prod_{P|L} (1+|P|^{-1})^{-1} + O\left(q^{(1+\epsilon) n}\right).
$$
\end{proof}

Applying Lemma \ref{auxiliary lemmas-2} with $\epsilon = \frac{1}2$, we have the following corollary.

\begin{cor}\label{auxiliary lemmas-3}
We have
\begin{align}\label{auxiliary lemmas-3-1}
\sum_{L\in\A_{l}^{+}} \sum_{\substack{f\in\A_{n}^{+}\\ (f,L)=1}} \Phi(f)
= \frac{1}{\zeta_{\A}(2)} q^{2n+l} \sum_{\substack{D\in\A^{+}\\ \deg(D)\le l}} \frac{\mu(D)}{|D|} \prod_{P|D} \frac{1}{|P|+1}
+ O\left(q^{\frac{3n}{2}+l}\right).
\end{align}
\end{cor}
\begin{proof}
By Lemma \ref{auxiliary lemmas-2}, we have
\begin{align*}
\sum_{L\in\A_{l}^{+}} \sum_{\substack{f\in\A_{n}^{+}\\ (f,L)=1}} \Phi(f)
&= \frac{1}{\zeta_{\A}(2)} q^{2n} \sum_{L\in\A_{l}^{+}} \prod_{P|L}(1+|P|^{-1})^{-1} + O\left(q^{\frac{3n}{2}+l}\right),
\end{align*}
and, by \cite[Lemma 5.7]{AK12},
$$
\sum_{L\in\A_{l}^{+}} \prod_{P|L}(1+|P|^{-1})^{-1} = q^{l} \sum_{\substack{D\in\A^{+}\\ \deg(D)\le l}} \frac{\mu(D)}{|D|} \prod_{P|D} \frac{1}{|P|+1}.
$$
So we get the result.
\end{proof}

Now we have the following lemma, which is a generalization of Lemmas 3.3, 3.4 and 3.5 of \cite{Ju14}.

\begin{lem}\label{auxiliary lemmas-3-2}
Let $h\in\{g-1, g\}$.
For any $s\in\mb C$ with ${\rm Re}(s)\ge \frac{1}2$, we have
\begin{enumerate}
\item
\begin{align}\label{auxiliary lemmas-3-2-001}
\sum_{\substack{D\in\A^{+}\\ \deg[D]\le[\frac{h}2]}}\frac{\mu(D)}{|D|^{2s}} \prod_{P|D} \frac{1}{|P|+1} = {\rm P}(2s) + O(q^{-sg}),
\end{align}
\item
\begin{align}\label{auxiliary lemmas-3-2-002}
\hspace{0.33em}\sum_{\substack{D\in\A^{+}\\ \deg(D)\le[\frac{h}2]}} \frac{\mu(D)}{|D|^{2-2s}}\prod_{P|D} \frac{1}{|P|+1} =
\begin{cases}
{\rm P}(2-2s) + O(q^{(s-1)g}) & \text{ if\,\,${\rm Re}(s) < 1$}, \\
O(g) & \text{ if\,\,$s=1$}, \\
O(q^{(s-1)g}) & \text{ if\,\,${\rm Re}(s) \ge  1 ~~(s\ne 1)$. }
\end{cases}
\end{align}
\end{enumerate}
\end{lem}
\begin{proof}
(1) We can write
\begin{align}\label{auxiliary lemmas-3-001}
\sum_{\substack{D\in\A^{+}\\ \deg(D)\le[\frac{h}2]}}\frac{\mu(D)}{|D|^{2s}} \prod_{P|D} \frac{1}{|P|+1}
&= \sum_{D\in\A^{+}}\frac{\mu(D)}{|D|^{2s}} \prod_{P|D} \frac{1}{|P|+1}  \nonumber \\
&\hspace{2em}- \sum_{\substack{D\in\A^{+}\\ \deg(D)>[\frac{h}2]}}\frac{\mu(D)}{|D|^{2s}} \prod_{P|D} \frac{1}{|P|+1}.
\end{align}
Using the Euler product formula, we can show that
\begin{align*}
\sum_{D\in\A^{+}}\frac{\mu(D)}{|D|^{s}} \prod_{P|D} \frac{1}{|P|+1} = \prod_{P}\left(1-\frac{1}{|P|^s (|P|+1)}\right) = {\rm P}(s).
\end{align*}
Hence, we have
\begin{align}\label{auxiliary lemmas-3-002}
\sum_{D\in\A^{+}}\frac{\mu(D)}{|D|^{2s}} \prod_{P|D} \frac{1}{|P|+1} = {\rm P}(2s).
\end{align}
We also have
\begin{align}\label{auxiliary lemmas-3-003}
\sum_{\substack{D\in\A^{+}\\ \deg(D)>[\frac{h}2]}}\frac{\mu(D)}{|D|^{2s}} \prod_{P|D} \frac{1}{|P|+1}
&\ll \sum_{\substack{D\in\A^{+}\\ \deg(D)>[\frac{h}2]}}\frac{\mu(D)^2}{|D|^{2s}} \prod_{P|D} \frac{1}{|P|}
\ll \sum_{\substack{D\in\A^{+}\\ \deg(D)>[\frac{h}2]}} |D|^{-(2s+1)} \nonumber\\
&= \sum_{n>[\frac{h}2]} \sum_{D\in\A_{n}^{+}} q^{-(2s+1)n} = \sum_{n>[\frac{h}2]} q^{-2sn} \ll q^{-sg}.
\end{align}
By inserting \eqref{auxiliary lemmas-3-002} and \eqref{auxiliary lemmas-3-003} into \eqref{auxiliary lemmas-3-001}, we get \eqref{auxiliary lemmas-3-2-001}.

\noindent(2) If ${\rm Re}(s) < 1$, by a similar process as in the proof of (1), we can get
$$
\sum_{\substack{D\in\A^{+}\\ \deg(D)\le[\frac{h}2]}}\frac{\mu(D)}{|D|^{2-2s}} \prod_{P|D} \frac{1}{|P|+1} = {\rm P}(2-2s) + O\left(q^{(s-1)g}\right).
$$
If $s=1$, we have
\begin{align*}
\sum_{\substack{D\in\A^{+}\\ \deg(D)\le[\frac{h}2]}} \mu(D) \prod_{P|D} \frac{1}{|P|+1}
\le \sum_{\substack{D\in\A^{+}\\ \deg(D)\le[\frac{h}2]}} \mu^2(D) \prod_{P|D} \frac{1}{|P|}
\le \sum_{\substack{D\in\A^{+}\\ \deg(D)\le[\frac{h}2]}} |D|^{-1} \le g.
\end{align*}
Finally, if ${\rm Re}(s) \ge  1 ~~(s\ne 1)$, we have
\begin{align*}
\sum_{\substack{D\in\A^{+}\\ \deg(D)\le[\frac{h}2]}}\frac{\mu(D)}{|D|^{2-2s}} \prod_{P|D} \frac{1}{|P|+1}
&\ll \sum_{\substack{D\in\A^{+}\\ \deg(D)\le[\frac{h}2]}} \mu^{2}(D) |D|^{2s-2} \prod_{P|D} \frac{1}{|P|} \\
&\ll \sum_{\substack{D\in\A^{+}\\ \deg(D)\le[\frac{h}2]}} |D|^{2s-3}
= \sum_{n=0}^{[\frac{h}2]} q^{(2s-2)n} \ll q^{(s-1)g}.
\end{align*}
\end{proof}

Recall that, for a small $\varepsilon>0$,
$$
\bar{X}_{\varepsilon} = \{\tfrac{1}2\} \cup {X}_{\varepsilon} \cup \{s\in \mb C : {\rm Re}(s) \ge 1\},
$$
where ${X}_{\varepsilon} = \{s \in \mb C : \tfrac{1}2 \le {\rm Re}(s) < 1, |s-\tfrac{1}2|>\varepsilon\}$.
For $h\in\{g-1, g\}$ and $s \in\bar{X}_{\varepsilon}$, consider the following two summations
\begin{equation}\label{Hgs}
J_{h}(s) = \sum_{l=0}^{[\frac{h}2]} q^{(1-2s)l} \sum_{\substack{D\in\A^{+}\\ \deg(D)\le l}} \frac{\mu(D)}{|D|} \prod_{P|D} \frac{1}{|P|+1}
\end{equation}
and
\begin{equation}\label{Hgss}
\tilde{J}_{h}(s) = \sum_{l=0}^{[\frac{h}2]} q^{(2s-1)l} \sum_{\substack{D\in\A^{+}\\ \deg(D)\le l}} \frac{\mu(D)}{|D|} \prod_{P|D} \frac{1}{|P|+1},
\end{equation}
which will be used repeatedly in \S\ref{C-S-F-M}.
Recall that (see (5.29) in \cite{AK12})
\begin{equation}\label{AK12-5-29}
{\rm P'}(s) = {\rm P}(s) \log q \sum_{P\in\mb P} \frac{\deg(P)}{|P|^{s}(|P|+1)-1}.
\end{equation}

\begin{lem}\label{auxiliary lemmas-4}
Let $h\in\{g-1, g\}$.
For a small $\varepsilon>0$ and $s\in {X}_{\varepsilon} \cup \{s\in \mb C : {\rm Re}(s) \ge 1\}$, we have
\begin{align}\label{auxiliary lemmas-4-001}
J_{h}(s) = \zeta_{\A}(2s) \left({\rm P}(2s) - q^{(1-2s)([\frac{h}2]+1)} {\rm P}(1)\right) + O\left(q^{-sg}\right),
\end{align}
and for $s=\frac{1}2$, we have
\begin{align}\label{auxiliary lemmas-4-001}
J_{h}(\tfrac{1}2) = \left([\tfrac{h}2]+1+\frac{1}{\log q} \frac{{\rm P}'}{\rm P}(1)\right) {\rm P}(1) + O\left(g q^{-\frac{g}2}\right).
\end{align}
\end{lem}
\begin{proof}
By change of summations, we have
\begin{align*}
J_{h}(s) = \sum_{\substack{D\in\A^{+}\\ \deg(D) \le [\frac{h}2]}} \frac{\mu(D)}{|D|} \prod_{P|D} \frac{1}{|P|+1}
\sum_{\deg(D) \le m \le [\frac{h}2]} q^{(1-2s)m}.
\end{align*}
If $s=\frac{1}2$, by \eqref{auxiliary lemmas-3-2-001} and \cite[Lemma 5.11, Proposition 5.12]{AK12}, we have
\begin{align*}
J_{h}(\tfrac{1}2) &= ([\tfrac{h}2]+1) \sum_{\substack{D\in\A^{+}\\ \deg(D) \le [\frac{h}2]}} \frac{\mu(D)}{|D|} \prod_{P|D} \frac{1}{|P|+1}
- \sum_{\substack{D\in\A^{+}\\ \deg(D) \le [\frac{h}2]}} \frac{\mu(D) \deg(D)}{|D|} \prod_{P|D} \frac{1}{|P|+1} \\
&= \left([\tfrac{h}2]+1+\frac{1}{\log q} \frac{{\rm P}'}{\rm P}(1)\right) {\rm P}(1) + O\left(g q^{-\frac{g}2}\right).
\end{align*}
For $s\in {X}_{\varepsilon} \cup \{s\in \mb C : {\rm Re}(s) \ge 1\}$, $\zeta_{\A}(2s)$ is bounded.
Now replacing $\sum_{\deg(D) \le m \le [\frac{h}2]} q^{(1-2s)m}$ by
$$
\frac{q^{(1-2s)\deg(D)} - q^{(1-2s)([\frac{h}2]+1)}}{1-q^{1-2s}}
= \zeta_{\A}(2s) \left(q^{(1-2s)\deg(D)} - q^{(1-2s)([\frac{h}2]+1)}\right),
$$
we have, by Lemma \ref{auxiliary lemmas-3-2} (1),
\begin{align*}
J_{h}(s) &= \zeta_{\A}(2s) \sum_{\substack{D\in\A^{+}\\ \deg(D)\le [\frac{h}2]}} \frac{\mu(D)}{|D|^{2s}} \prod_{P|D} \frac{1}{|P|+1} \nonumber \\
&\hspace{2em} - \zeta_{\A}(2s) q^{(1-2s)([\frac{h}2]+1)}
\sum_{\substack{D\in\A^{+}\\ \deg(D)\le [\frac{h}2]}} \frac{\mu(D)}{|D|} \prod_{P|D} \frac{1}{|P|+1} \nonumber \\
&= \zeta_{\A}(2s) \left({\rm P}(2s) - q^{(1-2s)([\frac{h}2]+1)} {\rm P}(1)\right) + O(q^{-sg}).
\end{align*}
\end{proof}

\begin{lem}\label{auxiliary lemmas-5}
Let $h\in\{g-1, g\}$.
For a small $\varepsilon>0$ and $s\in\bar{X}_{\varepsilon}$, we have
\begin{align*}
\tilde{J}_{h}(s) = \begin{cases}
\big([\frac{h}{2}]+1+\frac{1}{\log q} \frac{{\rm P}'}{\rm P}(1)\big) {\rm P}(1) + O(g q^{-\frac{g}2}) & \text{ if $s=\frac{1}2$,} \vspace{0.4em}\\
\zeta_{\A}(2s)\big(q^{(2s-1)[\frac{h}{2}]} {\rm P}(1) - q^{(1-2s)} {\rm P}(2-2s)\big) + O(q^{(s-1)g})
& \text{ if $s\in X_{\varepsilon}$,}\vspace{0.4em} \\
\zeta_{\A}(2s) q^{(2s-1)[\frac{h}{2}]} {\rm P}(1) + O(g q^{(s-1)g}) & \text{ if ${\rm Re}(s) \ge 1$.}
\end{cases}
\end{align*}
\end{lem}
\begin{proof}
We can write
\begin{align*}
\tilde{J}_{h}(s) = \sum_{\substack{D\in\A^{+}\\ \deg(D) \le [\frac{h}2]}} \frac{\mu(D)}{|D|} \prod_{P|D} \frac{1}{|P|+1}
\sum_{\deg(D) \le m \le [\frac{h}2]} q^{(2s-1)m}.
\end{align*}
If $s=\frac{1}2$, by \eqref{auxiliary lemmas-3-2-001} and \cite[Lemma 5.11, Proposition 5.12]{AK12}, we have
\begin{align*}
\tilde{J}_{h}(\tfrac{1}2) &= ([\tfrac{h}2]+1) \sum_{\substack{D\in\A^{+}\\ \deg(D) \le [\frac{h}2]}} \frac{\mu(D)}{|D|} \prod_{P|D} \frac{1}{|P|+1}
- \sum_{\substack{D\in\A^{+}\\ \deg(D) \le [\frac{h}2]}} \frac{\mu(D) \deg(D)}{|D|} \prod_{P|D} \frac{1}{|P|+1}  \\
&= \left([\tfrac{h}2]+1+\frac{1}{\log q} \frac{{\rm P}'}{\rm P}(1)\right) {\rm P}(1) + O\left(g q^{-\frac{g}2}\right).
\end{align*}
For $s \ne \frac{1}2$, as in the proof of Lemma \ref{auxiliary lemmas-4}, we have
\begin{align*}
\tilde{J}_{h}(s) &= \zeta_{\A}(2s) q^{(2s-1)[\frac{h}2]}
\sum_{\substack{D\in\A^{+}\\ \deg(D)\le [\frac{h}2]}} \frac{\mu(D)}{|D|} \prod_{P|D} \frac{1}{|P|+1} \nonumber \\
&\hspace{2em} - \zeta_{\A}(2s) q^{(1-2s)}\sum_{\substack{D\in\A^{+}\\ \deg(D)\le [\frac{h}2]}} \frac{\mu(D)}{|D|^{2-2s}} \prod_{P|D} \frac{1}{|P|+1}.
\end{align*}
Now, the results follows from \eqref{auxiliary lemmas-3-2-001} and \eqref{auxiliary lemmas-3-2-002}.
\end{proof}

\subsection{Two sums}
For each $M\in\mc B$, let $\mc C_{M}$ be the set of rational functions $u \in \mc C$ whose denominator divides $M$,
$\mc E_{M} = \mc E \cap \mc C_{M}$ and $\mc F_{M} = \mc F \cap \mc C_{M}$.
Note that $\mc C_{M}$ and $\mc E_{M}$ are abelian groups under addition and $\#\mc E_{M} = |\tilde{M}|$,
$\#\mc F_{M} = \Phi(\tilde{M})$.
In fact, $\mc C_{M}$ can be defined for any monic polynomial $M$, not necessarily in $\mc B$.
Note that $\mc E_{n}$ (resp. $\mc F_{n}$) is the disjoint union of $\mc E_{M}$ (resp. $\mc F_{M}$) with $M\in\mc B_{n}$.

For any $f\in\A^{+}$ and $M\in\mc B$, we define
$$
U_{f, M} := \sum_{u\in\mc C_{M}} \left\{\frac{u}{f}\right\}, \quad \Gamma_{f, M} := \sum_{u\in\mc E_{M}} \left\{\frac{u}{f}\right\}
\quad \text{ and }\quad T_{f, M} := \sum_{u\in\mc F_{M}} \left\{\frac{u}{f}\right\}.
$$

\begin{lem}\label{Two sums-1}
Let $f\in\A^{+}$ and $M\in\mc B$.
If $\deg(f) \le \deg(t(M))$, $\gcd(M,f)=1$ and $f$ is not a perfect square, then we have $U_{f,M} = \Gamma_{f,M} = 0$.
\end{lem}
\begin{proof}
It is easy to see that if $\deg(f) \le \deg(t(M))$, then $\mc C_{t(M)}$ contains a complete system modulo $f$.
Note that every element in $\mc E_{M}$ can be obtained by normalization from an element of $\mc C_{t(M)}$
and the normalization process preserves the value $\{\frac{u}{f}\}$.
Then, since $\mc C_{t(M)}$ and $\mc E_{M}$ are abelian groups, we get the result.
\end{proof}

Note that if $\gcd(M,f) = 1$ and there is a $u\in\mc E_{M}$ with $\{\frac{u}{f}\} = -1$, then $\Gamma_{f, M} = 0$.
Thus, we have $\Gamma_{f, M} = 0$ or $|\tilde{M}| = q^{\deg(t(M))/2}$.

We have
\begin{align*}
\Gamma_{f,M} = \sum_{\substack{N\in\mc B \\ N|M}} T_{f,N}
\end{align*}
and thus
\begin{align*}
T_{f,M} = \sum_{N|\tilde{M}} \mu(N) \Gamma_{f,[M/N]},
\end{align*}
where $[M/N] = ({\tilde{M}}/{N})^{*}$ for $N|\tilde{M}$.
Recall that $\tilde{M} := \sqrt{t(M)}$ and $M^* := {M^2}/{r(M)}$.

\begin{lem}\label{lem-333}
Let $M, N \in\mc B$ with $N|M$.
If $\Gamma_{f,N}=0$, then $\Gamma_{f,M}=0$.
Thus, if $\Gamma_{f,M} \ne 0$, then $\Gamma_{f,N} \ne 0$.
\end{lem}

\begin{lem}\label{lem-444}
Let $M, N \in \mc B$.
If $\gcd(M,N)=1$, then $\Gamma_{f, MN} = \Gamma_{f, M} \Gamma_{f, N}$.
Thus, for $M = P_{1}^{e_{1}} \cdots P_{r}^{e_{r}} \in \mc B$, we have $\Gamma_{f, M} = \Gamma_{f, P_{1}^{e_{1}}} \cdots \Gamma_{f, P_{r}^{e_{r}}}$.
\end{lem}
\begin{proof}
This follows from the fact that $u\in \mc E_{MN}$ can be uniquely written as $u_{1} + u_{2}$ with $u_{1} \in \mc E_{M}$ and $u_{2} \in \mc E_{N}$
and the fact that $\{\frac{u_{1}+u_{2}}{f}\} = \{\frac{u_{1}}{f}\} \{\frac{u_{2}}{f}\}$.
\end{proof}

\begin{cor}
For each positive integer $m$ and a non-square $f \in \A_{d}^{+}$,
there exist at most ${{[d/2]}\choose m}$ polynomials $M\in\A_{m}^{+}$ with $\Gamma_{f, M^*} \ne 0$ and $\gcd(M,f)=1$.
\end{cor}
\begin{proof}
Let $M_{1}, \ldots, M_{s} \in \A_{m}^{+}$ be distinct polynomials such that $\Gamma_{f, M_{i}^{*}} \ne 0$, $1 \le i \le s$.
For $M := {\rm lcm}(M_{1}, \ldots, M_{s})$, by Lemma \ref{lem-333} and Lemma \ref{lem-444}, we have $\Gamma_{f, M^*} \ne 0$, and so $2 \deg(M) < d$.
Note that if $s>{{[d/2]}\choose m}$, then $\deg(M) > m + [\frac{d}2] - m$, which implies that $2 \deg(M) \ge d$ and so $\Gamma_{f, M^*} = 0$.
Thus we should have $s \le {{[d/2]}\choose m}$ and the result follows.
\end{proof}

Now the following proposition is immediate, which can be viewed as an analogue of \cite[Lemma 6.2]{AK12}.

\begin{prop}\label{prop666}
Let $f \in \A_{d}^{+}$ be non-square and $m$ be a positive integer with $2m<d$.
Then we have
$$
\sum_{u\in\mc E_{m}(f)} \left\{\frac{u}{f}\right\} = \sum_{M\in\A_{m}^{+}} \Gamma_{f, M^*} \le {[d/2]\choose m} q^{m},
$$
where $\mc E_{m}(f)$ is the set of elements $u \in \mc E_{m}$ such that $u \in \mc E_{M}$
for some $M \in \mc B_{m}$ with $\gcd(M,f) = 1$.
\end{prop}

\begin{lem}\label{Two sums-2}
Let $M\in\mc B$ and $f\in\A^{+}$ with $\deg(f) \le \deg(M)$.
\begin{enumerate}
\item
If $\gcd(f,M) \ne 1$, then $T_{f,M} = 0$.
\item
If $\gcd(f,M) = 1$ and $f$ is a perfect square, then $T_{f,M} = \Phi(\tilde{M})$.
\item
If $\gcd(f,M) = 1$ and $f$ is not a perfect square, then
$$
T_{f, M} = \sum_{\substack{N|\tilde{M}\\ \deg([M/N]) < \deg(f)}} \mu(N) \Gamma_{f, [M/N]}.
$$
\end{enumerate}
\end{lem}

Note that if $u\in\mc E_{g+1}$, then the genus of the curve defined by $X^2+X+u=0$ is $g$.

\begin{lem}
Let $n$ be a positive integer.
For any $L\in\mc B_{\ell}$ with $\ell< n$, there are $(1-q^{-1}) q^{n-\ell}$ square-free monic polynomials $N \in \A^{+}$
such that there exists $M\in\mc B_{n}$ with $[M/N]=L$.
\end{lem}
\begin{proof}
It follows easily from the fact that $t(M) = t([M/N])N^2$ if $N$ is square-free with $N|\tilde{M}$.
\end{proof}

\begin{prop}\label{Non-square-Real}
Let $n$ be a positive integer.
For any monic polynomial $f \in \A^{+}_{d}$, which is not a perfect square, we have
\begin{align*}
\sum_{u\in\mc F_{n}} \left\{\frac{u}{f}\right\} \ll 2^{\frac{d}2} q^{n}.
\end{align*}
\end{prop}
\begin{proof}
Let $L\in\mc B_{\ell}$ with $2\ell<d$.
For each square-free monic polynomial $N$ of degree $n-\ell$,
there exists a unique $M\in\mc B_{n}$ such that $L=[M/N]$ and vice versa.
Then, by Lemmas \ref{Two sums-1} and \ref{Two sums-2}, we have
\begin{align*}
\sum_{u\in\mc F_{n}} \left\{\frac{u}{f}\right\}
&= \sum_{\substack{M\in\mc B_{n}\\ (M,f)=1}} T_{f,M}
= \sum_{\ell=0}^{[\frac{d}2]-1} \sum_{\substack{L\in\mc B_{\ell}\\ (L,f)=1}}
\sum_{\substack{N\in\A_{n-\ell}^{+}\\ (N,f)=1}} \mu(N) \Gamma_{f, L} \\
&\le \sum_{\ell=0}^{[\frac{d}2]-1} (1-q^{-1}) q^{n-\ell} \sum_{\substack{L\in\mc B_{\ell}\\ (L,f)=1}} \Gamma_{f, L} \\
&\le \sum_{\ell=0}^{[\frac{d}2]-1} (1-q^{-1}) q^{n-\ell} {[d/2]\choose \ell} q^{\ell} \ll 2^{\frac{d}2} q^{n}.
\end{align*}
\end{proof}

Now we consider the ramified imaginary case.
Let $s$ be a positive integer.
For $M\in\mc B$ and $\alpha \in \bF_{q}^{\times}$, we define (inside the field $k$)
\begin{align*}
& \mc C'_{M,s} := \mc C_{t(M)} + \A_{\le 2s-2},  \hspace{2.1em}  \mc C_{M,s,\alpha} := \mc C'_{M,s} + \alpha T^{2s-1}, \\
&\mc E'_{M,s} := \mc E_{M} + \tilde{\mc G}_{s-1}, \hspace{4em}  \mc E_{M,s,\alpha} := \mc E'_{M,s} + \alpha T^{2s-1}, \\
&\mc F'_{M,s} := \mc F_{M} + \tilde{\mc G}_{s-1}, \hspace{3.5em}  \mc F_{M,s,\alpha} := \mc F'_{M,s} + \alpha T^{2s-1},
\end{align*}
where $\A_{\le m} = \{f \in \A : \deg(f) \le m\}$, $\tilde{\mc G}_{0} = \{0\}$ and $\tilde{\mc G}_{s-1}$ denotes the set of all polynomials
$F(T) \in \A_{\le 2s-1}$ of the form $F(T) = \sum_{i=1}^{s-1} \alpha_{i} T^{2i-1}$ for $s \ge 2$.
For $f \in \A^{+}$, we also define
\begin{align*}
&\tilde{U}'_{f,M,s} := \sum_{u \in \mc C'_{M,s}} \left\{\frac{u}{f}\right\}, \hspace{2.3em}
\tilde{\Gamma}'_{f,M,s} := \sum_{u \in \mc E'_{M,s}} \left\{\frac{u}{f}\right\}, \hspace{2.5em}
\tilde{T}'_{f,M,s} := \sum_{u \in \mc F'_{M,s}} \left\{\frac{u}{f}\right\},  \\
&\tilde{U}_{f,M,s,\alpha} := \sum_{u \in \mc C_{M,s,\alpha}} \left\{\frac{u}{f}\right\}, \hspace{1em}
\tilde{\Gamma}_{f,M,s,\alpha} := \sum_{u \in \mc E_{M,s,\alpha}} \left\{\frac{u}{f}\right\}, \hspace{1em}
\tilde{T}_{f,M,s,\alpha} := \sum_{u \in \mc F_{M,s,\alpha}} \left\{\frac{u}{f}\right\},  \\
&\tilde{U}_{f,M,s} := \sum_{\alpha\in\bF_{q}^{\times}} \tilde{U}_{f,M,s,\alpha}, \hspace{2em}
\tilde{\Gamma}_{f,M,s} := \sum_{\alpha\in\bF_{q}^{\times}} \tilde{\Gamma}_{f,M,s,\alpha}, \hspace{2.2em}
\tilde{T}_{f,M,s} := \sum_{\alpha\in\bF_{q}^{\times}} \tilde{T}_{f,M,s,\alpha}.
\end{align*}

\begin{lem}
Let $r\ge 0$ and $s \ge 1$ be integers, $M\in\mc B_{r}$ and $\alpha\in\bF_{q}^{\times}$.
For any $f \in \A_{m}^+$ with $\gcd(M,f)=1$ and $m\le 2r+2s-2$, which is not a perfect square, we have
\begin{align*}
\tilde{U}'_{f,M,s} = \tilde{U}_{f,M,s,\alpha} = \tilde{\Gamma}'_{f,M,s} = \tilde{\Gamma}_{f,M,s,\alpha} = 0.
\end{align*}
\end{lem}
\begin{proof}
If $\deg(A) < \deg(f) \le 2r+2s-2$, then $A$ can be expressed as $t(M) B + C$ with $\deg(C) < 2r$ and $\deg(B) = \deg(A) - 2r \le 2s-2$.
Thus $\mc C'_{M,s}$ contains a complete residue system modulo $f$.
Also, every element in $\mc E'_{M,s}$ can be obtained from $\mc C'_{M,s}$ by normalization.
Since $\mc C'_{M,s}$ and $\mc E'_{M,s}$ are abelian groups, we have $\tilde{U}'_{f,M,s} = \tilde{\Gamma}'_{f,M,s} = 0$.
Since $\tilde{U}_{f,M,s,\alpha} = \{\frac{\alpha T^{2s-1}}{f}\}\tilde{U}'_{f,M,s}$ and
$\tilde{\Gamma}_{f,M,s,\alpha} = \{\frac{\alpha T^{2s-1}}{f}\}\tilde{\Gamma}'_{f,M,s}$,
we also have $\tilde{U}_{f,M,s,\alpha} = \tilde{\Gamma}_{f,M,s,\alpha} = 0$.
\end{proof}

As before we have
\begin{align*}
\tilde{\Gamma}'_{f,M,s} = \sum_{\substack{N\in\mc B\\ N|M}} \tilde{T}'_{f,N,s}
\end{align*}
and thus
$$
\tilde{T}'_{f,M,s} = \sum_{N|\tilde{M}} \mu(N) \tilde{\Gamma}'_{f,[M/N],s}.
$$

\begin{lem}\label{nonsquare-ramified-lemma1}
Let $M,N\in\mc B$ with $N|M$.
If $\tilde{\Gamma}'_{f,N,s} = 0$, then $\tilde{\Gamma}'_{f,M,s} = 0$.
\end{lem}

\begin{lem}\label{nonsquare-ramified-lemma1}
Let $s$ be a positive integer, $M \in \mc B$ and $f\in A^{+}$ with $\deg(f) \le \deg(t(M))+2s-2$.
\begin{enumerate}
\item
If $\gcd(f,M) \ne 1$, then $\tilde{T}'_{f,M,s} = 0$.
\item
If $\gcd(f,M) = 1$ and $f$ is a perfect square, then $\tilde{T}'_{f,M,s} = \Phi(\tilde{M}) q^{s-1}$.
\item
If $\gcd(f,M) = 1$ and $f$ is not a perfect square, then
$$
\tilde{T}'_{f,M,s} = \sum_{\substack{N|\tilde{M}\\ \deg(t([M/N]))<\deg(f)-2s+2}} \mu(N) \tilde{\Gamma}'_{f,[M/N],s}.
$$
\end{enumerate}
\end{lem}

For any positive integers $r,s$, let $\mc E_{(r,s)} = \bigcup_{M\in\mc B_{r}} \mc E'_{M,s}$.
Note that if $\Gamma_{f, M} = 0$, then $\tilde{\Gamma}'_{f,M,s}=0$ for any $s\ge 1$.
Then, using similar process as in Proposition \ref{prop666}, we get

\begin{lem}\label{nonsquare-ramified-lemma2}
Let $r\ge 0$ and $s \ge 1$ be integers.
For any $f\in\A^{+}_{d}$ with $d > 2r+2s-2$, which is not a perfect square, we have
$$
\sum_{u\in\mc E_{(r,s)}(f)} \left\{\frac{u}{f}\right\} \le {[d/2]-s \choose r} q^{r+s},
$$
where $\mc E_{(r,s)}(f)$ is the set of elements $u \in \mc E_{(r,s)}$ such that $u \in \mc E'_{M,s}$
for some $M \in \mc B_{r}$ with $\gcd(M,f) = 1$.
\end{lem}

\begin{prop}\label{nonsquare-ramified}
Let $r\ge 0$ and $s \ge 1$ be integers.
For any $f\in\A^{+}_{d}$ with $d \le g = r+s-1$, which is not a perfect square, we have
$$
\sum_{u\in{\mc I}_{(r,s)}} \left\{\frac{u}{f}\right\} \le 2^{\frac{d}2-s} q^{r+s}.
$$
Thus we have
\begin{align}\label{nonsquare-ramified-001}
\sum_{u\in{\mc I}_{g+1}} \left\{\frac{u}{f}\right\} \ll g 2^{\frac{d}2} q^{g}.
\end{align}
\end{prop}
\begin{proof}
Write
\begin{align*}
\sum_{u\in{\mc I}_{(r,s)}} \left\{\frac{u}{f}\right\}
= \sum_{\alpha\in\bF_{q}^{*}} \left\{\frac{\alpha T^{2s-1}}{f}\right\} \left(\sum_{M\in\mc B_{r}} \tilde{T}'_{f,M,s}\right)
\end{align*}
and, as in the proof of Proposition \ref{Non-square-Real} using
Lemmas \ref{nonsquare-ramified-lemma1} and \ref{nonsquare-ramified-lemma2}, we have
\begin{align*}
\sum_{\substack{M\in\mc B_{r} \\ (M,f)=1}} \tilde{T}'_{f,M,s}
&= \sum_{\ell=0}^{[\frac{d}2]-s} \sum_{\substack{L\in\mc B_{\ell} \\ (L,f)=1}}
\sum_{\substack{N\in\A_{r-\ell}^{+}\\ (N,f)=1}} \mu(N) \tilde{\Gamma}'_{f,L,s}\\
&\le \sum_{\ell=0}^{[\frac{d}2]-s} (1-q^{-1}) {[d/2]-s\choose \ell} q^{r-\ell} q^{\ell+s} \le 2^{\frac{d}2-s} q^{r+s}.
\end{align*}
Now, \eqref{nonsquare-ramified-001} follows from
$$
\sum_{u\in{\mc I}_{g+1}} \left\{\frac{u}{f}\right\} = \sum_{s=1}^{g+1} \sum_{u\in\mc I_{(g+1-s,s)}} \left\{\frac{u}{f}\right\}.
$$
\end{proof}

\section{Proof of Theorem \ref{mean-1}}\label{S-4}
In this section we give a proof of Theorem \ref{mean-1}.
In \S\ref{C-S-F-M} and \S\ref{C-N-S-F-M}, we obtain several results on the contribution of squares and of non-squares,
which will be used to calculate the first moment of $L$-functions for ${\rm Re(s)} \ge \frac{1}2$
in \S\ref{subsect3-3}, \S\ref{subsect3-4} and \S\ref{subsect3-5}.

\subsection{The Main Term: Contribution of squares}\label{C-S-F-M}
For a small $\varepsilon>0$ and $s\in\bar{X}_{\varepsilon}$, let
\begin{align}\label{Ag}
A_{g}(s) = \begin{cases}
\frac{{\rm P}(1)}{\zeta_{\A}(2)}\big([\tfrac{g}2]+1+\frac{1}{\log q} \frac{{\rm P}'}{\rm P}(1)\big) & \text{ if $s=\frac{1}2$,} \vspace{0.5em}\\
\frac{\zeta_{\A}(2s)}{\zeta_{\A}(2)} \big({\rm P}(2s) - q^{(1-2s)([\frac{g}2]+1)} {\rm P}(1)\big) & \text{ if $s\in {X}_{\varepsilon}$,} \vspace{0.5em}\\
\frac{\zeta_{\A}(2s)}{\zeta_{\A}(2)} {\rm P}(2s) & \text{ if ${\rm Re}(s) \ge 1$,}
\end{cases}
\end{align}
and
\begin{align}\label{Bg}
B_{g}(s) = \begin{cases}
\frac{{\rm P}(1)}{\zeta_{\A}(2)}  \left([\tfrac{g-1}2]+1+\frac{1}{\log q} \frac{{\rm P}'}{\rm P}(1)\right) & \text{ if $s = \frac{1}2$,}\vspace{0.5em} \\
\frac{\zeta_{\A}(2s)}{\zeta_{\A}(2)} q^{(1-2s)g} \left(q^{(2s-1)[\frac{g-1}2]} {\rm P}(1) - q^{(1-2s)} {\rm P}(2-2s)\right)
& \text{ if $s\in {X}_{\varepsilon}$,}\vspace{0.5em} \\
0 & \text{ if ${\rm Re}(s) \ge 1$.}
\end{cases}
\end{align}

\begin{prop}\label{S-C-Ramified-001}
For a small $\varepsilon>0$ and $s\in\bar{X}_{\varepsilon}$, we have
\begin{enumerate}
\item
$\dis \sum_{u\in\mc I_{g+1}}\sum_{n=0}^{g} q^{-sn} \sum_{\substack{f\in\A_{n}^{+}\\f=\square}}\left\{\frac{u}{f}\right\}
= 2 A_{g}(s) q^{2g+1} + O\left(g q^{\frac{3g}2}\right)$,
\item
$\dis q^{(1-2s)g} \sum_{u\in\mc I_{g+1}} \sum_{n=0}^{g-1} q^{(s-1)n} \sum_{\substack{f\in\A_{n}^{+}\\f=\square}}\left\{\frac{u}{f}\right\}
= 2 B_{g}(s) q^{2g+1} + O\left(g q^{\frac{3g}2}\right)$.
\end{enumerate}
\end{prop}
\begin{proof}
(1) Write
$$
\sum_{u\in\mc I_{g+1}}\sum_{n=0}^{g} q^{-sn} \sum_{\substack{f\in\A_{n}^{+}\\f=\square}}\left\{\frac{u}{f}\right\}
= \sum_{r=0}^{g} \sum_{u\in\mc I_{(r,g+1-r)}}\sum_{n=0}^{g} q^{-sn}
\sum_{\substack{f\in\A_{n}^{+}\\f=\square}}\left\{\frac{u}{f}\right\}.
$$
Note that $\mc I_{(0,g+1)} = \mc G_{g+1}$.
For $1 \le r \le g$ and $M\in\mc B_{r}$, let $\mc I_{M} = \{v+F : v \in \mc F_{M} \text{ and } F\in\mc G_{g+1-r}\}$.
Then $\mc I_{(r,g+1-r)}$ is the disjoint union of the $\mc I_{M}$'s, where $M$ runs over $\mc B_{r}$.
Hence we have
\begin{align*}
&\sum_{r=0}^{g} \sum_{u\in\mc I_{(r,g+1-r)}}\sum_{n=0}^{g} q^{-sn}
\sum_{\substack{f\in\A_{n}^{+}\\f=\square}}\left\{\frac{u}{f}\right\} \nonumber \\
&\hspace{1em}= \sum_{l=0}^{[\frac{g}2]} q^{-2sl} \sum_{L\in\A_{l}^{+}} \sum_{F\in\mc G_{g+1}} \left\{\frac{F}{L^2}\right\}
+ \sum_{l=0}^{[\frac{g}2]} q^{-2sl} \sum_{L\in\A_{l}^{+}} \sum_{r=1}^{g}
\sum_{M\in\mc B_{r}} \sum_{u\in\mc I_{M}} \left\{\frac{u}{L^2}\right\} \nonumber \\
&\hspace{1em}= \sum_{l=0}^{[\frac{g}2]} q^{-2sl} \sum_{L\in\A_{l}^{+}} \sum_{F\in\mc G_{g+1}} 1
+ \sum_{l=0}^{[\frac{g}2]} q^{-2sl} \sum_{L\in\A_{l}^{+}} \sum_{r=1}^{g}
\sum_{\substack{M\in\mc B_{r}\\(M,L)=1}} \sum_{u\in\mc I_{M}} 1.
\end{align*}
Since $\#\mc G_{g+1} = 2 \zeta_{\A}(2)^{-1} q^{g+1}$, we have
\begin{align*}
\sum_{l=0}^{[\frac{g}2]} q^{-2sl} \sum_{L\in\A_{l}^{+}} \sum_{F\in\mc G_{g+1}} 1
= 2 \zeta_{\A}(2)^{-1} q^{g+1} \sum_{l=0}^{[\frac{g}2]} q^{(1-2s)l} \ll g q^{g}.
\end{align*}
Since $\#\mc I_{M} = \frac{2}{\zeta_{\A}(2)} q^{g+1-r} \Phi(\tilde{M})$, by \eqref{auxiliary lemmas-3-1}, we have
\begin{align}\label{S-C-Ramified-001-1}
\sum_{l=0}^{[\frac{g}2]} q^{-2sl} \sum_{L\in\A_{l}^{+}} \sum_{r=1}^{g} \sum_{\substack{M\in\mc B_{r}\\(M,L)=1}} \sum_{u\in\mc I_{M}} 1
&= \frac{2}{\zeta_{\A}(2)} q^{g+1} \sum_{l=0}^{[\frac{g}2]} q^{-2sl} \sum_{L\in\A_{l}^{+}}
\sum_{r=1}^{g} q^{-r} \sum_{\substack{\tilde{M}\in\A_{r}^{+}\\(\tilde{M},L)=1}} \Phi(\tilde{M}) \nonumber \\
&\hspace{-3em}= \frac{2}{\zeta_{\A}(2)} (q^{g}-1) q^{g+1} J_{g}(s) + O\left(q^{\frac{3g}{2}} \sum_{l=0}^{[\frac{g}2]} q^{(1-2s)l}\right),
\end{align}
where $J_{g}(s)$ is given in \eqref{Hgs}.
The error term in \eqref{S-C-Ramified-001-1} is $\ll g q^{\frac{3g}{2}}$.
Now, the result follows from Lemma \ref{auxiliary lemmas-4}.

\noindent(2)
By a similar process as in the proof of (1), we have
\begin{align*}
q^{(1-2s)g} \sum_{u\in\mc I_{g+1}}\sum_{n=0}^{g-1} q^{(s-1)n}
\sum_{\substack{f\in\A_{n}^{+}\\f=\square}}\left\{\frac{u}{f}\right\}
&= q^{(1-2s)g} \sum_{F\in\mc G_{g+1}}\sum_{n=0}^{g-1} q^{(s-1)n} \sum_{\substack{f\in\A_{n}^{+}\\f=\square}}\left\{\frac{F}{f}\right\}  \\
&\hspace{1em} + q^{(1-2s)g} \sum_{r=0}^{g} \sum_{u\in\mc I_{(r,g+1-r)}}\sum_{n=0}^{g-1} q^{(s-1)n}
\sum_{\substack{f\in\A_{n}^{+}\\f=\square}}\left\{\frac{u}{f}\right\} \\
&\hspace{-5em}= q^{(1-2s)g} \sum_{l=0}^{[\frac{g-1}2]} q^{2(s-1)l} \sum_{L\in\A_{l}^{+}} \sum_{r=1}^{g}
\sum_{\substack{\tilde{M}\in\A_{r}^{+}\\(\tilde{M},L)=1}} \Phi(\tilde{M}) + O\left(g q^{(\frac{3}2-s)g}\right)\\
&\hspace{-5em}= \frac{2}{\zeta_{\A}(2)} (q^{g}-1) q^{(2-2s)g+1} \tilde{J}_{g-1}(s) + O\left(g q^{\frac{3g}{2}}\right),
\end{align*}
where $\tilde{J}_{g-1}(s)$ is given in \eqref{Hgss}.
The result follows from Lemma \ref{auxiliary lemmas-5}.
\end{proof}

\begin{prop}\label{S-C-Real-001}
For a small $\varepsilon>0$ and $s\in\bar{X}_{\varepsilon}$, we have
\begin{enumerate}
\item
$\dis \sum_{u\in\mc F_{g+1}}\sum_{n=0}^{g} (\pm 1)^{n} q^{-sn} \sum_{\substack{f\in\A_{n}^{+}\\f=\square}}\left\{\frac{u}{f}\right\}
= {A}_{g}(s) q^{2g+2}+ O\left(g q^{\frac{3g}{2}}\right)$,

\item
$\dis q^{(1-2s)g} \sum_{u\in\mc F_{g+1}} \sum_{n=0}^{g-1} (\pm 1)^{n} q^{(s-1)n}
\sum_{\substack{f\in\A_{n}^{+}\\f=\square}}\left\{\frac{u}{f}\right\} = {B}_{g}(s) q^{2g+2} + O\left(gq^{\frac{3g}2}\right)$.
\end{enumerate}
\end{prop}
\begin{proof}
(1) For each $M\in\mc B_{g+1}$, let $\mc F_{M}$ be the set of rational functions $u \in \mc F_{g+1}$ whose denominator is $M$.
Then $\mc F_{g+1}$ is a disjoint union of the $\mc F_{M}$'s, where $M$ runs over $\mc B_{g+1}$.
Hence, we can write
\begin{align*}
\sum_{u\in\mc F_{g+1}}\sum_{n=0}^{g} (\pm 1)^{n} q^{-sn} \sum_{\substack{f\in\A_{n}^{+}\\f=\square}}\left\{\frac{u}{f}\right\}
&= \sum_{n=0}^{g} (\pm 1)^{n} q^{-sn} \sum_{\substack{f\in\A_{n}^{+}\\f=\square}} \sum_{M\in\mc B_{g+1}} \sum_{u\in\mc F_{M}}\left\{\frac{u}{f}\right\} \\
&= \sum_{l=0}^{[\frac{g}2]} q^{-2sl} \sum_{L\in\A_{l}^{+}} \sum_{M\in\mc B_{g+1}} \sum_{u\in\mc F_{M}} \left\{\frac{u}{L^2}\right\} \\
&= \sum_{l=0}^{[\frac{g}2]} q^{-2sl} \sum_{L\in\A_{l}^{+}} \sum_{\substack{M\in\mc B_{g+1}\\ (M,L)=1}} \sum_{u\in\mc F_{M}} 1.
\end{align*}
Since $\#\mc F_{M} = \Phi(\tilde{M})$, by \eqref{auxiliary lemmas-3-1}, we have
\begin{align}\label{S-C-Real-001-1}
\sum_{l=0}^{[\frac{g}2]} q^{-2sl} \sum_{L\in\A_{l}^{+}} \sum_{\substack{M\in\mc B_{g+1}\\ (M,L)=1}} \sum_{u\in\mc F_{M}} 1
&= \sum_{l=0}^{[\frac{g}2]} q^{-2sl} \sum_{L\in\A_{l}^{+}} \sum_{\substack{\tilde{M}\in\A_{g+1}^{+}\\ (\tilde{M},L)=1}} \Phi(\tilde{M}) \nonumber \\
&= \frac{1}{\zeta_{\A}(2)} q^{2g+2} J_{g}(s)
+ O\left(q^{\frac{3g}{2}} \sum_{l=0}^{[\frac{g}2]} q^{(1-2s)l}\right).
\end{align}
The error term in \eqref{S-C-Real-001-1} is $\ll g q^{\frac{3g}{2}}$,
and the result follows from Lemma \ref{auxiliary lemmas-4}.

\noindent (2) By a similar process as in the proof of (1), we have
\begin{align}\label{S-C-Real-001-2}
& q^{(1-2s)g} \sum_{u\in\mc F_{g+1}} \sum_{n=0}^{g-1} (\pm 1)^{n} q^{(s-1)n} \sum_{\substack{f\in\A_{n}^{+}\\f=\square}}\left\{\frac{u}{f}\right\}\nonumber  \\
&\hspace{1em}= q^{(1-2s)g} \sum_{n=0}^{g-1} (\pm 1)^{n} q^{(s-1)n} \sum_{\substack{f\in\A_{n}^{+}\\f=\square}}
\sum_{M\in\mc B_{g+1}}\sum_{u\in\mc F_{M}} \left\{\frac{u}{f}\right\}\nonumber \\
&\hspace{1em}= q^{(1-2s)g} \sum_{l=0}^{[\frac{g-1}2]} q^{2(s-1)l} \sum_{L\in\A_{l}^{+}}
\sum_{\substack{\tilde{M}\in\A_{g+1}^{+}\\ (\tilde{M},L)=1}} \Phi(\tilde{M}) \nonumber\\
&\hspace{1em}= \frac{1}{\zeta_{\A}(2)} \tilde{J}_{g-1}(s) q^{(3-2s)g+2} + O\left(q^{(\frac{5}2-2s)g}\sum_{l=0}^{[\frac{g-1}2]} q^{(2s-1)l}\right),
\end{align}
where the error term in \eqref{S-C-Real-001-2} is $\ll g q^{\frac{3g}{2}}$.
Now, the result follows from Lemma \ref{auxiliary lemmas-5}.
\end{proof}

\begin{prop}\label{S-C-Real-002}
Let $h\in\{g-1, g\}$.
For $s\in\mb C$ with ${\rm Re}(s)\ge \frac{1}2$, we have
$$
q^{-(h+1)s} \sum_{u\in\mc F_{g+1}}\sum_{n=0}^{h} (\pm 1)^{n} \sum_{\substack{f\in\A_{n}^{+}\\ f=\square}} \left\{\frac{u}{f}\right\}
= {\rm P}(1) q^{2g+2-(h+1)s+[\frac{h}2]} + O\left(g q^{(2-s)g}\right).
$$
\end{prop}
\begin{proof}
By a similar process as in the proof of (1) of Proposition \ref{S-C-Real-001}, we have
\begin{align*}
q^{-(h+1)s} \sum_{u\in\mc F_{g+1}}\sum_{n=0}^{h} (\pm 1)^{n} \sum_{\substack{f\in\A_{n}^{+}\\ f=\square}} \left\{\frac{u}{f}\right\}
&= q^{-(h+1)s} \sum_{n=0}^{h} (\pm 1)^{n} \sum_{\substack{f\in\A_{n}^{+}\\ f=\square}}
\sum_{M\in\mc B_{g+1}}\sum_{u\in\mc F_{M}} \left\{\frac{u}{f}\right\} \\
&= q^{-(h+1)s} \sum_{l=0}^{[\frac{h}2]} \sum_{L\in\A_{l}^{+}} \sum_{M\in\mc B_{g+1}}\sum_{u\in\mc F_{M}} \left\{\frac{u}{L^2}\right\} \\
&= q^{-(h+1)s} \sum_{l=0}^{[\frac{h}2]} \sum_{L\in\A_{l}^{+}} \sum_{\substack{\tilde{M}\in\A_{g+1}^{+}\\ (\tilde{M},L)=1}} \Phi(\tilde{M}) \\
&= \frac{1}{\zeta_{\A}(2)} q^{2g+2-(h+1)s} \tilde{J}_{h}(1) + O\left(q^{(2-s)g}\right).
\end{align*}
Now, the result follows from Lemma \ref{auxiliary lemmas-5}.
\end{proof}

\begin{prop}\label{S-C-Real-003}
For $s\in\mb C$ with ${\rm Re}(s) \ge \frac{1}2$, we have
\begin{align*}
q^{-g} \sum_{u\in\mc F_{g+1}}\sum_{n=0}^{g-1} (g-n) \sum_{\substack{f\in\A_{n}^{+}\\ f=\square}}\left\{\frac{u}{f}\right\}
= O\left(gq^{\frac{3g}2}\right).
\end{align*}
\end{prop}
\begin{proof}
By a similar process as in the proof of Proposition \ref{S-C-Real-001} (1),
and by Corollary \ref{auxiliary lemmas-3}, we have
\begin{align*}
q^{-g} \sum_{u\in\mc F_{g+1}}\sum_{n=0}^{g-1} (g-n) \sum_{\substack{f\in\A_{n}^{+}\\ f=\square}}\left\{\frac{u}{f}\right\}
&= q^{-g} \sum_{n=0}^{g-1} (g-n) \sum_{\substack{f\in\A_{n}^{+}\\ f=\square}} \sum_{M\in\mc B_{g+1}}\sum_{u\in\mc F_{M}} \left\{\frac{u}{f}\right\} \\
&= q^{-g} \sum_{l=0}^{[\frac{g-1}2]} (g-2l) \sum_{L\in\A_{l}^{+}} \sum_{M\in\mc B_{g+1}}\sum_{u\in\mc F_{M}} \left\{\frac{u}{L^2}\right\} \\
&= q^{-g} \sum_{l=0}^{[\frac{g-1}2]} (g-2l) \sum_{L\in\A_{l}^{+}} \sum_{\substack{\tilde{M}\in\A_{g+1}^{+}\\ (\tilde{M},L)=1}} \Phi(\tilde{M}) \\
&\hspace{-5em}= \frac{1}{\zeta_{\A}(2)} q^{g+1} \sum_{l=0}^{[\frac{g-1}2]} (g-2l) q^{l}
\sum_{\substack{D\in\A^{+}\\ \deg(D)\le l}} \frac{\mu(D)}{|D|} \prod_{P|D}\frac{1}{|P|+1} + O\left(q^{g}\right).
\end{align*}
Moreover, we have
\begin{align*}
q^{g+1} \sum_{l=0}^{[\frac{g-1}2]} (g-2l) q^{l}\sum_{\substack{D\in\A^{+}\\ \deg(D)\le l}} \frac{\mu(D)}{|D|} \prod_{P|D}\frac{1}{|P|+1}
&\ll q^{g} \sum_{l=0}^{[\frac{g-1}2]} (g-2l) q^{l} \sum_{\substack{D\in\A^{+}\\ \deg(D)\le l}} \frac{\mu^2(D)}{|D|} \prod_{P|D}\frac{1}{|P|} \\
&\ll q^{g} \sum_{l=0}^{[\frac{g-1}2]} (g-2l) q^{l} \sum_{\substack{D\in\A^{+}\\ \deg(D)\le l}} |D|^{-2} \\
&\ll q^{g} \sum_{l=0}^{[\frac{g-1}2]} l (g-2l) q^{l} \ll g q^{\frac{3g}2}.
\end{align*}
Hence we get the result.
\end{proof}

\subsection{The Error Term: Contributions of non-squares}\label{C-N-S-F-M}
Let $s_{0} = 1 + \frac{1}2 \log_{q} 2$.

\begin{prop}\label{N-S-C-Ramified-001}
For $s\in\mb C$ with ${\rm Re}(s) \ge \frac{1}2$, we have
\begin{enumerate}
\item
$\dis \sum_{u\in\mc I_{g+1}}\sum_{n=0}^{g} q^{-sn} \sum_{\substack{f\in\A_{n}^{+}\\f\ne\square}}\left\{\frac{u}{f}\right\}
\ll \begin{cases}
g 2^{\frac{g}2} q^{(2-s)g} & \text{ if ${\rm Re}(s)<s_{0}$,} \\
g^2 q^{g} & \text{ if ${\rm Re}(s) \ge s_{0}$,}
\end{cases}$
\item
$\dis q^{(1-2s)g} \sum_{u\in\mc I_{g+1}} \sum_{n=0}^{g-1} q^{(s-1)n}
\sum_{\substack{f\in\A_{n}^{+}\\f\ne\square}}\left\{\frac{u}{f}\right\} \ll g 2^{\frac{g}2} q^{(2-s)g}$.
\end{enumerate}
\end{prop}
\begin{proof}
(1) By \eqref{nonsquare-ramified-001}, we have
\begin{align*}
\sum_{u\in\mc I_{g+1}}\sum_{n=0}^{g} q^{-sn} \sum_{\substack{f\in\A_{n}^{+}\\f\ne\square}}\left\{\frac{u}{f}\right\}
&\ll \sum_{n=0}^{g} q^{-sn} \sum_{\substack{f\in\A_{n}^{+}\\f\ne\square}} \left|\sum_{u\in\mc I_{g+1}} \left\{\frac{u}{f}\right\}\right| \\
&\ll g q^{g} \sum_{n=0}^{g} (2^{1/2} q^{1-s})^{n}  \\
&\ll \begin{cases}
g 2^{\frac{g}2} q^{(2-s)g} & \text{ if ${\rm Re}(s)<s_{0}$,} \\
g^2 q^{g} & \text{ if ${\rm Re}(s) \ge s_{0}$.}
\end{cases}
\end{align*}
The proof of (2) is similar to that of (1) and there is no novelty involved.
\end{proof}

\begin{prop}\label{N-S-C-Real-001}
For $s\in\mb C$ with ${\rm Re}(s) \ge \frac{1}2$, we have
\begin{enumerate}
\item
$\dis \sum_{u\in\mc F_{g+1}}\sum_{n=0}^{g} (\pm 1)^{n} q^{-sn} \sum_{\substack{f\in\A_{n}^{+}\\f\ne\square}}\left\{\frac{u}{f}\right\}
\ll \begin{cases}
2^{\frac{g}2} q^{(2-s)g} & \text{ if ${\rm Re}(s)<s_{0}$,} \\
g q^{g} & \text{ if ${\rm Re}(s) \ge s_{0}$,}
\end{cases}$
\item
$\dis q^{(1-2s)g} \sum_{u\in\mc F_{g+1}} \sum_{n=0}^{g-1} (\pm 1)^{n} q^{(s-1)n}
\sum_{\substack{f\in\A_{n}^{+}\\f\ne\square}}\left\{\frac{u}{f}\right\} \ll 2^{\frac{g}2} q^{(2-s)g}$,
\item
$\dis q^{-(h+1)s} \sum_{u\in\mc F_{g+1}}\sum_{n=0}^{h} (\pm 1)^{n}
\sum_{\substack{f\in\A_{n}^{+}\\ f\ne\square}} \left\{\frac{u}{f}\right\} \ll 2^{\frac{g}2} q^{(2-s)g}$, where $h\in\{g-1, g\}$,
\item
$\dis q^{-g} \sum_{u\in\mc F_{g+1}}\sum_{n=0}^{g-1} (g-n)
\sum_{\substack{f\in\A_{n}^{+}\\ f\ne\square}}\left\{\frac{u}{f}\right\} \ll 2^{\frac{g}2} q^{g}$.
\end{enumerate}
\end{prop}
\begin{proof}
(1) By Proposition \ref{Non-square-Real}, we have
\begin{align*}
\sum_{u\in\mc F_{g+1}}\sum_{n=0}^{g} (\pm 1)^{n} q^{-sn} \sum_{\substack{f\in\A_{n}^{+}\\ f\ne\square}}\left\{\frac{u}{f}\right\}
&\ll \sum_{n=0}^{g} q^{-sn} \sum_{\substack{f\in\A_{n}^{+}\\ f\ne\square}}
\left|\sum_{u\in\mc F_{g+1}} \left\{\frac{u}{f}\right\}\right| \\
&\ll q^{g} \sum_{n=0}^{g} (2^{\frac{1}2} q^{1-s})^{n} \\
&\ll \begin{cases}
2^{\frac{g}2} q^{(2-s)g} & \text{ if ${\rm Re}(s)<s_{0}$,} \\
g q^{g} & \text{ if ${\rm Re}(s) \ge s_{0}$.}
\end{cases}
\end{align*}
The proofs of (2), (3) and (4) are similar to that of (1) and there is no novelty involved.
\end{proof}

\begin{rem}\label{rem-s1}
Let $s_{1} = \frac{1}2(1+\log_{q}2)$.
For $s\in \mb C$ with ${\rm Re(s)} \ge \frac{1}2$, we have
\begin{align*}
q^{\frac{3g}2} \ll 2^{\frac{g}2} q^{(2-s)g} \quad\text{ if ${\rm Re}(s) < s_{1}$}
\end{align*}
and
\begin{align*}
q^{\frac{3g}2} \gg 2^{\frac{g}2} q^{(2-s)g} \quad\text{ if ${\rm Re}(s) \ge s_{1}$.}
\end{align*}
\end{rem}

\subsection{Proof of Theorem \ref{mean-1} (1)}\label{subsect3-3}
By Lemma \ref{E-A-FE} (1), we have
\begin{align}\label{proof-ramified-001}
\sum_{u\in\mc I_{g+1}} L(s,\chi_{u})
= \sum_{u\in\mc I_{g+1}}\sum_{n=0}^{g} q^{-sn} \sum_{f\in\A_{n}^{+}}\chi_{u}(f)
+ q^{(1-2s)g} \sum_{u\in\mc I_{g+1}} \sum_{n=0}^{g-1} q^{(s-1)n} \sum_{f\in\A_{n}^{+}}\chi_{u}(f).
\end{align}
For a small $\varepsilon>0$ and $s\in\bar{X}_{\varepsilon}$, by Propositions \ref{S-C-Ramified-001}, \ref{N-S-C-Ramified-001} and Remark \ref{rem-s1}, we have
\begin{align}\label{proof-ramified-002}
\sum_{u\in\mc I_{g+1}}\sum_{n=0}^{g} q^{-sn} \sum_{f\in\A_{n}^{+}}\chi_{u}(f)
&= \sum_{u\in\mc I_{g+1}}\sum_{n=0}^{g} q^{-sn} \sum_{\substack{f\in\A_{n}^{+}\\ f = \square}}\chi_{u}(f)
+ \sum_{u\in\mc I_{g+1}}\sum_{n=0}^{g} q^{-sn} \sum_{\substack{f\in\A_{n}^{+}\\ f \ne \square}} \chi_{u}(f) \nonumber \\
&= 2 A_{g}(s) q^{2g+1} + \begin{cases}
O(g 2^{\frac{g}2} q^{(2-s)g}) & \text{ if ${\rm Re}(s) < s_{1}$,} \\
O(g q^{\frac{3g}2}) & \text{ if ${\rm Re}(s) \ge s_{1}$,}
\end{cases}
\end{align}
and
\begin{align}\label{proof-ramified-003}
q^{(1-2s)g} \sum_{u\in\mc I_{g+1}} \sum_{n=0}^{g-1} q^{(s-1)n} \sum_{f\in\A_{n}^{+}}\chi_{u}(f)
&= q^{(1-2s)g} \sum_{u\in\mc I_{g+1}} \sum_{n=0}^{g-1} q^{(s-1)n} \sum_{\substack{f\in\A_{n}^{+}\\ f = \square}}\chi_{u}(f) \nonumber \\
&\hspace{2em}+ q^{(1-2s)g} \sum_{u\in\mc I_{g+1}} \sum_{n=0}^{g-1} q^{(s-1)n} \sum_{\substack{f\in\A_{n}^{+}\\ f \ne \square}}\chi_{u}(f) \nonumber \\
&\hspace{-6em}= 2 B_{g}(s) q^{2g+1} + \begin{cases}
O(g 2^{\frac{g}2} q^{(2-s)g}) & \text{ if ${\rm Re}(s) < s_{1}$,} \\
O(g q^{\frac{3g}2}) & \text{ if ${\rm Re}(s) \ge s_{1}$.}
\end{cases}
\end{align}
By inserting \eqref{proof-ramified-002} and \eqref{proof-ramified-003} into \eqref{proof-ramified-001}, we have
\begin{align*}
\sum_{u\in\mc I_{g+1}} L(s,\chi_{u}) = 2 (A_{g}(s)+B_{g}(s)) q^{2g+1} + \begin{cases}
O(g 2^{\frac{g}2} q^{(2-s)g}) & \text{ if ${\rm Re}(s) < s_{1}$,} \\
O(g q^{\frac{3g}2}) & \text{ if ${\rm Re}(s) \ge s_{1}$,}
\end{cases}
\end{align*}
where $A_{g}(s)+B_{g}(s)$ equals to
\begin{align*}
\alpha_{g}(s) = \begin{cases}
\frac{{\rm P}(1)}{\zeta_{\A}(2)}\big(g+1+\frac{2}{\log q} \frac{{\rm P}'}{\rm P}(1)\big) & \text{ if $s=\frac{1}2$,} \vspace{0.5em} \\
\frac{\zeta_{\A}(2s)}{\zeta_{\A}(2)} \big\{{\rm P}(2s) - q^{(1-2s)(g+1)} {\rm P}(2-2s)   \\
 \hspace{3.5em} + {\rm P}(1)\big(q^{(1-2s)(g-[\frac{g-1}2])} - q^{(1-2s)([\frac{g}2]+1)}\big)\big\}
& \text{ if $s\in X_{\varepsilon}$,}\vspace{0.5em} \\
\frac{\zeta_{\A}(2s)}{\zeta_{\A}(2)} {\rm P}(2s) & \text{ if ${\rm Re}(s) \ge 1$.}
\end{cases}
\end{align*}

\subsection{Proof of Theorem \ref{mean-1} (2)}\label{subsect3-4}
By Lemma \ref{E-A-FE} (2), we have
\begin{align}\label{proof-real-001}
\sum_{u\in\mc F_{g+1}}L(s,\chi_{u})
&= \sum_{u\in\mc F_{g+1}}\sum_{n=0}^{g} q^{-sn} \sum_{f\in\A_{n}^{+}}\chi_{u}(f) \nonumber \\
&\hspace{2em}- q^{-(g+1)s} \sum_{u\in\mc F_{g+1}}\sum_{n=0}^{g} \sum_{f\in\A_{n}^{+}}\chi_{u}(f) + \sum_{u\in\mc F_{g+1}} H_{u}(s),
\end{align}
where
$$
\sum_{u\in\mc F_{g+1}} H_{u}(1) = \zeta_{\A}(2)^{-1} q^{-g} \sum_{u\in\mc F_{g+1}} \sum_{n=0}^{g-1} \left(g-n\right)
\sum_{f\in\A_{n}^{+}}\left\{\frac{u}{f}\right\}
$$
and, for $s\ne 1$,
\begin{align*}
\sum_{u\in\mc F_{g+1}} H_{u}(s)
&= q^{(1-2s)g} \eta(s) \sum_{u\in\mc F_{g+1}} \sum_{n=0}^{g-1} q^{(s-1)n}
\sum_{f\in\A_{n}^{+}}\left\{\frac{u}{f}\right\} - q^{-s g} \eta(s) \sum_{u\in\mc F_{g+1}} \sum_{n=0}^{g-1}
\sum_{f\in\A_{n}^{+}}\left\{\frac{u}{f}\right\}
\end{align*}
with $\eta(s) = \frac{\zeta_{\A}(2-s)}{\zeta_{\A}(1+s)}$.
For a small $\varepsilon>0$ and $s\in\bar{X}_{\varepsilon}$, by Propositions \ref{S-C-Real-001}, \ref{S-C-Real-002}, \ref{N-S-C-Real-001}
and Remark \ref{rem-s1}, we have
\begin{align}\label{proof-real-002}
\sum_{u\in\mc F_{g+1}}\sum_{n=0}^{g} q^{-sn} \sum_{f\in\A_{n}^{+}}\chi_{u}(f)
&= \sum_{u\in\mc F_{g+1}}\sum_{n=0}^{g} q^{-sn} \sum_{\substack{f\in\A_{n}^{+}\\ f = \square}} \chi_{u}(f)
+ \sum_{u\in\mc F_{g+1}}\sum_{n=0}^{g} q^{-sn} \sum_{\substack{f\in\A_{n}^{+}\\ f \ne \square}} \chi_{u}(f) \nonumber \\
&= A_{g}(s) q^{2g+2} + \begin{cases}
O(2^{\frac{g}2} q^{(2-s)g}) & \text{ if ${\rm Re}(s) < s_{1}$,} \\
O(g q^{\frac{3g}2}) & \text{ if ${\rm Re}(s) \ge s_{1}$,}
\end{cases}
\end{align}
and
\begin{align}\label{proof-real-003}
q^{-(g+1)s} \sum_{u\in\mc F_{g+1}}\sum_{n=0}^{g} \sum_{f\in\A_{n}^{+}}\chi_{u}(f)
&= q^{-(g+1)s} \sum_{u\in\mc F_{g+1}}\sum_{n=0}^{g} \sum_{\substack{f\in\A_{n}^{+}\\ f = \square}} \chi_{u}(f)  \nonumber  \\
&\hspace{2em}+ q^{-(g+1)s} \sum_{u\in\mc F_{g+1}}\sum_{n=0}^{g} \sum_{\substack{f\in\A_{n}^{+}\\ f \ne \square}}\chi_{u}(f) \nonumber  \\
&\hspace{-5em}= {\rm P}(1) q^{(2-s)g+[\frac{g}2]+2-s} + \begin{cases}
O(2^{\frac{g}2} q^{(2-s)g}) & \text{ if ${\rm Re}(s) < s_{1}$,} \\
O(g q^{\frac{3g}2}) & \text{ if ${\rm Re}(s) \ge s_{1}$.}
\end{cases}
\end{align}
If $s=1$, by Propositions \ref{S-C-Real-003} and \ref{N-S-C-Real-001} (4), we have
\begin{align}\label{proof-real-004}
\sum_{u\in\mc F_{g+1}} H_{u}(1)
&= \zeta_{\A}(2)^{-1} q^{-g} \sum_{u\in\mc F_{g+1}} \sum_{n=0}^{g-1} (g-n) \sum_{\substack{f\in\A_{n}^{+}\\ f = \square}}\chi_{u}(f) \nonumber \\
&\hspace{2em}+ \zeta_{\A}(2)^{-1} q^{-g} \sum_{u\in\mc F_{g+1}} \sum_{n=0}^{g-1} (g-n) \sum_{\substack{f\in\A_{n}^{+}\\ f \ne \square}}\chi_{u}(f)
= O\left(g q^{\frac{3g}2}\right).
\end{align}
By inserting \eqref{proof-real-002}, \eqref{proof-real-003} with $s=1$ and \eqref{proof-real-004} into \eqref{proof-real-001}, we have
\begin{align}
\sum_{u\in\mc F_{g+1}}L(1,\chi_{u}) &= {\rm P}(2) q^{2g+2}+ O\left(g q^{\frac{3g}2}\right).
\end{align}
Now, consider the case $s\ne 1$.
Let $\delta>0$ be arbitrary small.
For $s\in \bar{X}_{\varepsilon}$ with $|s-1|>\delta$, since $\eta(s)$ is bounded, by Propositions \ref{S-C-Real-001} and \ref{N-S-C-Real-001}, we have
\begin{align}\label{proof-real-005}
q^{(1-2s)g} \eta(s) \sum_{u\in\mc F_{g+1}} \sum_{n=0}^{g-1} q^{(s-1)n} \sum_{f\in\A_{n}^{+}}\left\{\frac{u}{f}\right\}
&= q^{(1-2s)g} \eta(s) \sum_{u\in\mc F_{g+1}} \sum_{n=0}^{g-1} q^{(s-1)n} \sum_{\substack{f\in\A_{n}^{+}\\ f = \square}}\left\{\frac{u}{f}\right\} \nonumber\\
&\hspace{2em}+ q^{(1-2s)g} \eta(s) \sum_{u\in\mc F_{g+1}} \sum_{n=0}^{g-1} q^{(s-1)n} \sum_{\substack{f\in\A_{n}^{+}\\ f \ne \square}}\left\{\frac{u}{f}\right\}\nonumber\\
&\hspace{-8em}= \eta(s) {B}_{g}(s) q^{2g+2} + \begin{cases}
O(2^{\frac{g}2} q^{(2-s)g}) & \text{ if ${\rm Re}(s) < s_{1}$,} \\
O(g q^{\frac{3g}2}) & \text{ if ${\rm Re}(s) \ge s_{1}$,}
\end{cases}
\end{align}
and, by Propositions \ref{S-C-Real-002} and \ref{N-S-C-Real-001},
\begin{align}\label{proof-real-006}
q^{-s g} \eta(s) \sum_{u\in\mc F_{g+1}} \sum_{n=0}^{g-1} \sum_{f\in\A_{n}^{+}}\left\{\frac{u}{f}\right\}
&= q^{-s g} \eta(s) \sum_{u\in\mc F_{g+1}} \sum_{n=0}^{g-1} \sum_{\substack{f\in\A_{n}^{+}\\ f = \square}}\left\{\frac{u}{f}\right\}  \nonumber\\
&\hspace{2em}+ q^{-s g} \eta(s) \sum_{u\in\mc F_{g+1}} \sum_{n=0}^{g-1} \sum_{\substack{f\in\A_{n}^{+}\\ f \ne \square}}\left\{\frac{u}{f}\right\}  \nonumber\\
&\hspace{-8em}= \eta(s) {\rm P}(1) q^{(2-s)g+[\frac{g-1}2]+2} + \begin{cases}
O(2^{\frac{g}2} q^{(2-s)g}) & \text{ if ${\rm Re}(s) < s_{1}$,} \\
O(g q^{\frac{3g}2}) & \text{ if ${\rm Re}(s) \ge s_{1}$.}
\end{cases}
\end{align}
By inserting \eqref{proof-real-002}, \eqref{proof-real-003}, \eqref{proof-real-005} and \eqref{proof-real-006}
into \eqref{proof-real-001}, we have
\begin{align}
\sum_{u\in\mc F_{g+1}}L(s,\chi_{u})
&= \beta_{g}(s) q^{2g+2} + \begin{cases}
O(2^{\frac{g}2} q^{(2-s)g}) & \text{ if ${\rm Re}(s) < s_{1}$,} \\
O(g q^{\frac{3g}2}) & \text{ if ${\rm Re}(s) \ge s_{1}$.}
\end{cases}
\end{align}

\subsection{Proof of Theorem \ref{mean-1} (3)}\label{subsect3-5}
Note that $\chi_{u}(f) = (-1)^{\deg(f)} \chi_{v}(f)$ for $u = v+\xi \in \mc F'_{g+1}$ with $v\in\mc F_{g+1}$.
Then, by Lemma \ref{E-A-FE} (3), we have
\begin{align}\label{proof-inert-001}
\sum_{u\in\mc F'_{g+1}} L(s,\chi_{u})
&= \sum_{u\in\mc F_{g+1}}\sum_{n=0}^{g} (-1)^{n} q^{-sn} \sum_{f\in\A_{n}^{+}}\chi_{u}(f)
+ (-1)^{g} q^{-(g+1)s} \sum_{u\in\mc F_{g+1}}\sum_{n=0}^{g} \sum_{f\in\A_{n}^{+}}\chi_{u}(f) \nonumber  \\
&\hspace{1em}+ \nu(s) q^{(1-2s)g} \sum_{u\in\mc F_{g+1}}
\sum_{n=0}^{g-1} (-1)^{n} q^{(s-1)n}\sum_{f\in\A_{n}^{+}}\chi_{u}(f) \nonumber  \\
&\hspace{1em}+ (-1)^{g+1} \nu(s) q^{-sg} \sum_{u\in\mc F_{g+1}}\sum_{n=0}^{g-1} \sum_{f\in\A_{n}^{+}}\chi_{u}(f)
\end{align}
with $\nu(s) = \frac{1+q^{-s}}{1+q^{s-1}}$.
For a small $\varepsilon>0$ and $s\in\bar{X}_{\varepsilon}$, by Propositions \ref{S-C-Real-001} and \ref{N-S-C-Real-001}, we have
\begin{align}\label{proof-inert-002}
\sum_{u\in\mc F_{g+1}}\sum_{n=0}^{g} (-1)^{n} q^{-sn} \sum_{f\in\A_{n}^{+}}\chi_{u}(f)
&= \sum_{u\in\mc F_{g+1}}\sum_{n=0}^{g} (-1)^{n} q^{-sn} \sum_{\substack{f\in\A_{n}^{+}\\ f = \square}}\chi_{u}(f)  \nonumber \\
&\hspace{2em}+ \sum_{u\in\mc F_{g+1}}\sum_{n=0}^{g} (-1)^{n} q^{-sn} \sum_{\substack{f\in\A_{n}^{+}\\ f \ne \square}} \chi_{u}(f) \nonumber \\
&= {A}_{g}(s) q^{2g+2}+ \begin{cases}
O(2^{\frac{g}2} q^{(2-s)g}) & \text{ if ${\rm Re}(s) < s_{1}$,} \\
O(g q^{\frac{3g}2}) & \text{ if ${\rm Re}(s) \ge s_{1}$,}
\end{cases}
\end{align}
and
\begin{align}\label{proof-inert-003}
& \nu(s) q^{(1-2s)g} \sum_{u\in\mc F_{g+1}} \sum_{n=0}^{g-1} (-1)^{n} q^{(s-1)n}\sum_{f\in\A_{n}^{+}}\chi_{u}(f) \nonumber \\
&\hspace{1em}= \nu(s) q^{(1-2s)g} \sum_{u\in\mc F_{g+1}} \sum_{n=0}^{g-1} (-1)^{n} q^{(s-1)n}\sum_{\substack{f\in\A_{n}^{+}\\ f = \square}}\chi_{u}(f) \nonumber \\
&\hspace{3em}+ \nu(s) q^{(1-2s)g} \sum_{u\in\mc F_{g+1}} \sum_{n=0}^{g-1} (-1)^{n} q^{(s-1)n}\sum_{\substack{f\in\A_{n}^{+}\\ f \ne \square}}\chi_{u}(f) \nonumber \\
&\hspace{1em}= \nu(s) {B}_{g}(s) q^{2g+2} + \begin{cases}
O(2^{\frac{g}2} q^{(2-s)g}) & \text{ if ${\rm Re}(s) < s_{1}$,} \\
O(g q^{\frac{3g}2}) & \text{ if ${\rm Re}(s) \ge s_{1}$.}
\end{cases}
\end{align}
By Propositions \ref{S-C-Real-002} and \ref{N-S-C-Real-001}, we have
\begin{align}\label{proof-inert-004}
(-1)^{g} q^{-(g+1)s} \sum_{u\in\mc F_{g+1}}\sum_{n=0}^{g} \sum_{f\in\A_{n}^{+}}\chi_{u}(f)
&= (-1)^{g} q^{-(g+1)s} \sum_{u\in\mc F_{g+1}}\sum_{n=0}^{g} \sum_{\substack{f\in\A_{n}^{+}\\ f = \square}}\chi_{u}(f) \nonumber \\
&\hspace{2em}+ (-1)^{g} q^{-(g+1)s} \sum_{u\in\mc F_{g+1}}\sum_{n=0}^{g} \sum_{\substack{f\in\A_{n}^{+}\\ f \ne \square}}\chi_{u}(f) \nonumber \\
&\hspace{-8em}= (-1)^{g} {\rm P}(1) q^{(2-s)g+[\frac{g}2]+2-s} + \begin{cases}
O(2^{\frac{g}2} q^{(2-s)g}) & \text{ if ${\rm Re}(s) < s_{1}$,} \\
O(g q^{\frac{3g}2}) & \text{ if ${\rm Re}(s) \ge s_{1}$,}
\end{cases}
\end{align}
and
\begin{align}\label{proof-inert-005}
(-1)^{g+1} \nu(s) q^{-sg} \sum_{u\in\mc F_{g+1}}\sum_{n=0}^{g-1} \sum_{f\in\A_{n}^{+}}\chi_{u}(f)
&= (-1)^{g+1} \nu(s) q^{-sg} \sum_{u\in\mc F_{g+1}}\sum_{n=0}^{g-1} \sum_{\substack{f\in\A_{n}^{+}\\ f = \square}}\chi_{u}(f) \nonumber \\
&\hspace{2em}+ (-1)^{g+1} \nu(s) q^{-sg} \sum_{u\in\mc F_{g+1}}\sum_{n=0}^{g-1} \sum_{\substack{f\in\A_{n}^{+}\\ f = \square}}\chi_{u}(f)  \nonumber \\
&\hspace{-12em}= (-1)^{g+1} \nu(s) {\rm P}(1) q^{(2-s)g+[\frac{g-1}2]+2} + \begin{cases}
O(g 2^{\frac{g}2} q^{(2-s)g}) & \text{ if ${\rm Re}(s) < s_{1}$,} \\
O(g q^{\frac{3g}2}) & \text{ if ${\rm Re}(s) \ge s_{1}$.}
\end{cases}
\end{align}
By inserting \eqref{proof-inert-002}, \eqref{proof-inert-003}, \eqref{proof-inert-004}
and \eqref{proof-inert-005} into \eqref{proof-inert-001}, we have
\begin{align*}
\sum_{u\in\mc F'_{g+1}}L(s,\chi_{u}) &=  \gamma_{g}(s) q^{2g+2} + \begin{cases}
O(2^{\frac{g}2} q^{(2-s)g}) & \text{ if ${\rm Re}(s) < s_{1}$,} \\
O(g q^{\frac{3g}2}) & \text{ if ${\rm Re}(s) \ge s_{1}$.}
\end{cases}
\end{align*}

\begin{rem}
With the results in \S\ref{S3-1} and \S\ref{S3-2}, we can modify the proofs in \S\ref{S-4}
by replacing $u\in \mc I_{g+1}$ (resp. $\mc F_{g+1}$, $\mc F'_{g+1}$) by $D\in\mc H_{2g+1}$ (resp. $\mc H_{2g+2}$, $\gamma \mc H_{2g+2}$),
and $\{\frac{u}{f}\}$ by $(\frac{D}{f})$ to show that Theorem \ref{mean-1} is also true in odd characteristic.
Here $\mc H_{n}$ is the set of all monic square-free polynomials of degree $n$, and $\gamma$ is a generator of $\bF_{q}^{\times}$.
\end{rem}



\begin{bibdiv}
\begin{biblist}

\bib{An12}{article}{
   author={Andrade, Julio},
   title={A note on the mean value of $L$-functions in function fields},
   journal={Int. J. Number Theory},
   volume={8},
   date={2012},
   number={7},
   pages={1725--1740},
}



\bib{AK12}{article}{
   author={Andrade, J. C.},
   author={Keating, J. P.},
   title={The mean value of $L(\frac12,\chi)$ in the hyperelliptic ensemble},
   journal={J. Number Theory},
   volume={132},
   date={2012},
   number={12},
   pages={2793--2816},
}

\bib{AK13}{article}{
   author={Andrade, Julio C.},
   author={Keating, Jonathan P.},
   title={Mean value theorems for $L$-functions over prime polynomials for
   the rational function field},
   journal={Acta Arith.},
   volume={161},
   date={2013},
   number={4},
   pages={371--385},
}

\bib{AK14}{article}{
   author={Andrade, J. C.},
   author={Keating, J. P.},
   title={Conjectures for the integral moments and ratios of $L$-functions
   over function fields},
   journal={J. Number Theory},
   volume={142},
   date={2014},
   pages={102--148},
}
		

\bib{Ch08}{article}{
   author={Chen, Yen-Mei J.},
   title={Average values of $L$-functions in characteristic two},
   journal={J. Number Theory},
   volume={128},
   date={2008},
   number={7},
   pages={2138--2158},
}

\bib{CY08}{article}{
   author={Chen, Yen-Mei J.},
   author={Yu, Jing},
   title={On class number relations in characteristic two},
   journal={Math. Z.},
   volume={259},
   date={2008},
   number={1},
   pages={197--216},
}

\bib{Gauss}{book}{
   author={Gauss, C. F.},
   title={Disquisitiones Arithmeticae},
   publisher={Yale University Press},
   date={1966},
}

\bib{GH85}{article}{
   author={Goldfeld, Dorian},
   author={Hoffstein, Jeffrey},
   title={Eisenstein series of ${1\over 2}$-integral weight and the mean value of real Dirichlet $L$-series},
   journal={Invent. Math.},
   volume={80},
   date={1985},
   number={2},
   pages={185--208},
}

\bib{Ha34}{article}{
   author={Hasse, Helmut},
   title={Theorie der relativ-zyklischen algebraischen Funktionenk\"orper, inbesondre bei endlichen Konstantenk\"orper},
   journal={J. Reine Angew. Math.},
   volume={172},
   date={1934},
   pages={37--54},
}


\bib{HR92}{article}{
   author={Hoffstein, Jeffrey},
   author={Rosen, Michael},
   title={Average values of $L$-series in function fields},
   journal={J. Reine Angew. Math.},
   volume={426},
   date={1992},
   pages={117--150},
}

\bib{HL10}{article}{
   author={Hu, Su},
   author={Li, Yan},
   title={The genus fields of Artin-Schreier extensions},
   journal={Finite Fields Appl.},
   volume={16},
   date={2010},
   number={4},
   pages={255--264},
}

\bib{Ju13}{article}{
   author={Jung, Hwanyup},
   title={Note on the mean value of $L(\frac{1}2,\chi)$ in the hyperelliptic
   ensemble},
   journal={J. Number Theory},
   volume={133},
   date={2013},
   number={8},
   pages={2706--2714},
}

\bib{Ju14}{article}{
   author={Jung, Hwanyup},
   title={A note on the mean value of $L(1,\chi)$ in the hyperelliptic
   ensemble},
   journal={Int. J. Number Theory},
   volume={10},
   date={2014},
   number={4},
   pages={859--874},
}
	

\bib{Ju81}{article}{
   author={Jutila, M.},
   title={On the mean value of $L({1\over 2},\,\chi )$ for real characters},
   journal={Analysis},
   volume={1},
   date={1981},
   number={2},
   pages={149--161},
}

\bib{Ro02}{book}{
   author={Rosen, Michael},
   title={Number theory in function fields},
   series={Graduate Texts in Mathematics},
   volume={210},
   publisher={Springer-Verlag},
   place={New York},
   date={2002},
   pages={xii+358},
}

\bib{RW15}{article}{
   author={Rubinstein, Michael O.},
   author={Wu, Kaiyu},
   title={Moments of zeta functions associated to hyperelliptic curves over finite fields},
   journal={Phil. Trans. R. Soc.},
   volume={373},
   date={2015},
   pages={20140307},
}

\bib{Si44}{article}{
   author={Siegel, C. L.},
   title={The average measure of quadratic forms with given determinant and signature},
   journal={Ann. Math.},
   volume={45},
   date={1944},
   pages={667--685},
}

\bib{St93}{book}{
   author={Stichtenoth, Henning},
   title={Algebraic function fields and codes},
   series={Universitext},
   publisher={Springer-Verlag, Berlin},
   date={1993},
}

\end{biblist}
\end{bibdiv}
\end{document}